\documentclass{amsart}
\usepackage{amsfonts,amsthm,amsmath}
\usepackage{amssymb}
\usepackage[usenames,dvipsnames]{color} 
\usepackage[utf8]{inputenc}
\usepackage{xypic} 
\usepackage{epic}
\usepackage[a4paper]{geometry}
\usepackage{hyperref}
\usepackage{marginnote}
\usepackage{tikz}
\usetikzlibrary{arrows, intersections,calc}
\usepackage{bm} 
\usepackage{enumerate} 

\theoremstyle{plain}
	\newtheorem{thm}{Theorem}
	
	\newtheorem{lemma}[thm]{Lemma}
	\newtheorem{prop}[thm]{Proposition}
	
\theoremstyle{remark}
	\newtheorem{remark}[thm]{Remark}

\theoremstyle{definition}
	\newtheorem{definition}[thm]{Definition}	
\title{N\'emethi's division algorithm for zeta-functions of plumbed 3-manifolds}
\author{Tam\'as L\'aszl\'o}
\address{BCAM - Basque Center for Applied Mathematics\\
Mazarredo, 14 E48009 Bilbao, Basque Country – Spain \\}
\email{tlaszlo@bcamath.org}

\author{Zsolt Szil\'agyi}
\address{Alfr\'ed R\'enyi Institute of Mathematics, Hungarian Academy of Sciences,  
1053 Budapest, Re\'altanoda u. 13-15,  Hungary.}
\email{szilagyi.zsolt@renyi.mta.hu}


\begin{document}
\maketitle

\pagestyle{myheadings} \markboth{{\normalsize
T. L\'aszl\'o and Zs. Szil\'agyi}}{{\normalsize  Division algorithm}}

\begin{abstract}
A polynomial counterpart of the Seiberg-Witten invariant associated with a negative definite plumbed $3$-manifold has been proposed by earlier work of the authors. It is provided by a special decomposition of the zeta-function defined by the combinatorics of the manifold. In this article we give an algorithm, based on multivariable Euclidean division of the zeta-function, for the explicit calculation of the polynomial, in particular for the  Seiberg--Witten invariant.
\end{abstract}

\section{Introduction}

\subsection{} The main motivation of the present article is to understand a multivariable division algorithm, proposed by A. N\'emethi (cf. \cite{Npers}, \cite{BN}), for the calculation of the normalized Seiberg--Witten invariant of a negative definite plumbed 3-manifold. The input is a multivariable zeta-function associated with the manifold and the output is a (Laurent) polynomial, called the polynomial part of the zeta-function. In particular, this is a polynomial `categorification' of the Seiberg--Witten invariant in the sense that the sum of its coefficients equals with the normalized Seiberg--Witten invariant. The polynomial part was defined by the authors in \cite{LSz} as a possible solution for the multivariable `polynomial- and negative-degree part' decomposition problem for the zeta-function (cf. \cite{BN,LN,LSz}, see \ref{ss:polSW}). 

The one-variable algorithm goes back to the work of Braun and N\'emethi \cite{BN}. In that case the polynomial part is simply given by a division principle. However, in general, we show that in order to recover the multivariable polynomial part of \cite{LSz} one constructs a polynomial by  division and then one has to consider  its terms with suitable multiplicity according to the corresponding exponents and the structure of the plumbing graph. 

In the sequel, we give some details about the algorithm and state further results of the present note.
%
%
%
%
%

\subsection{}\label{intro2} Let $M$ be a closed oriented plumbed 3-manifold associated with a connected negative definite plumbing graph $\Gamma$. Or, equivalently, $M$ is the link of a complex normal surface singularity, and $\Gamma$ is its dual resolution graph. Assume that $M$ is a rational homology sphere, ie. $\Gamma$ is a tree and all the plumbed surfaces have genus zero. Let $\mathcal{V}$ be the set of vertices of $\Gamma$, $\delta_v$ be the valency of a vertex $v\in\mathcal{V}$, and we distinguish the following subsets: the set of \emph{nodes} $\mathcal{N}:=\{n\in \mathcal{V}:\delta_n \geq 3\}$ and the set of \emph{ends} $\mathcal{E}=\{v\in \mathcal{V}:\delta_v= 1\}$.

We consider the plumbed 4-manifold $\widetilde{X}$ associated with $\Gamma$. Its second homology $L:=H_2(\widetilde{X},\mathbb{Z})$ is a lattice, freely generated by the classes of 2-spheres $\{E_v\}_{v\in\mathcal{V}}$, endowed with the nondegenerate negative definite intersection form $(,)$. The second cohomology $L':=H^2(\widetilde{X},\mathbb{Z})$ is the dual lattice, freely generated by the (anti)dual classes $\{E_v^*\}_{v\in\mathcal{V}}$, where we set $(E^*_v,E_w)=-\delta_{vw}$, the negative of the Kronecker delta. The intersection form embeds $L$ into $L'$ and $H:=L'/L\simeq H_1(M,\mathbb{Z})$. Denote the class of $l'\in L'$ in $H$ by $[l']$. We denote by $\mathfrak{sw}^{norm}_h(M)$ the normalized Seiberg--Witten invariants of $M$ indexed by the group elements $h\in H$, see \ref{ss:sw}. 

The \emph{multivariable zeta-function} associated with $M$ (or $\Gamma$) is defined by  
$$
f(\mathbf{t}) = \prod_{v\in \mathcal{V}}(1-\mathbf{t}^{E^{*}_{v}})^{\delta_{v}-2},
$$
where $\mathbf{t}^{l'}:=\prod_{v\in\mathcal{V}}t_v^{l_v}$ for any $l'=\sum_{v\in\mathcal{V}}l_vE_v\in L'$. One has a natural decomposition into its $h$-equivariant parts $f(\mathbf{t})=\sum_{h\in H} f_{h}(\mathbf{t})$, see \ref{ss:def}. By a result of \cite{LN}, for our purposes, one can reduce the variables of $f_h$ to the variables of the nodes of the graph. Therefore we restrict our discussions to the reduced zeta-functions defined by $f_h(\mathbf{t}_{\mathcal{N}}) = f_{h}(\mathbf{t})|_{t_v=1,v\notin \mathcal{N}}$. Here we introduce notation $\mathbf{t}_{\mathcal{N}}^{l'} := \prod_{n\in\mathcal{N}}t_n^{l_n}$. 

\subsection{}\label{ss:intro3} 
The multivariable polynomial part $P_h(\mathbf{t}_{\mathcal{N}})$ associated with $f_h(\mathbf{t}_{\mathcal{N}})$ (\cite{LSz}), 
 is mainly a combination of the one- and two-variable cases studied by \cite{BN} and \cite{LN} corresponding to the structure of the \emph{orbifold graph} $\Gamma^{orb}$.
The vertices of $\Gamma^{orb}$ are the nodes of $\Gamma$ and two of them are connected by an edge if the corresponding nodes in $\Gamma$ are connected by a path which consists only vertices with valency $\delta_v=2$. 
The main property reads as $P_h(1)=\mathfrak{sw}^{norm}_h(M)$, see \ref{ss:polSW}.

\subsection{Multivariable division algorithm}
\label{sec:alg}
 On $L\otimes \mathbb{Q}$ we consider the partial order: for any $l_1,l_2$ one writes $l_1 > l_2$ if
$l_1-l_2=\sum_{v\in \mathcal{V}} \ell_v E_v$ with all $\ell_v > 0$.
We introduce a multivariable division algorithm in \ref{ss:multidivision}, which provides a unique decomposition
 (Lemma \ref{lem:+dec})
$$
f_{h}(\mathbf{t}_{\mathcal{N}}) = P^{+}_{h}(\mathbf{t}_{\mathcal{N}}) + f^{neg}_{h}(\mathbf{t}_{\mathcal{N}}),
$$ 
where $P^{+}_{h} (\mathbf{t}_{\mathcal{N}})= \sum_{\beta}p_{\beta}\mathbf{t}_{\mathcal{N}}^{\beta}$ is a Laurent polynomial such that  $\beta\not <0$ for every monomial and $f_{h}^{neg}(\mathbf{t}_{\mathcal{N}})$ is a rational function with negative degree in $t_{n}$ for all $n\in \mathcal{N}$. 

In \ref{ss:multiplicity} we define a multiplicity function $\mathfrak{s}$ involving the structure of $\Gamma^{orb}$ and we show in Theorem \ref{lm-0} that the polynomial part $P_{h}(\mathbf{t}_{\mathcal{N}})$ can be computed from the quotient $P^{+}_{h}(\mathbf{t}_{\mathcal{N}})$ by taking its monomial terms with multiplicity $\mathfrak{s}$. More precisely,
$$
P_{h}(\mathbf{t}_{\mathcal{N}}) = \sum_{\beta}\mathfrak{s}(\beta) p_{\beta}\mathbf{t}_{\mathcal{N}}^{\beta}.
$$

\subsection{Comparisons} A consequence of the above algorithm (cf. Remark \ref{rk:poly-plus}(\ref{poly-plus-i})) is that in general $P_h$ is `thicker' than $P^{+}_{h}$, in the sense that $\mathfrak{s}(\beta)\geq 1$ for all the exponents $\beta$ of $P^{+}_{h}$. This motivates the study of their comparison on two different classes of graphs. 

In the first case we assume that $\Gamma^{orb}$ is a bamboo, that is, there are no vertices with valency greater or equal than $3$. Notice that most of the examples considered in the aforementioned articles were taken from this class. We prove in Theorem \ref{thm:bamboo} that for these graphs the two polynomials agree. Thus, the Seiberg--Witten invariant calculation is provided only by the division. 

The second class is defined by a topological criterion: they are the graphs of the 3-manifolds $S^3_{-p/q}(K)$ obtained by $(-p/q)$-surgery along the connected sum $K$ of some algebraic knots. We provide a concrete example of this class for which one has $P_{h}\neq P^{+}_{h}$ for some $h$, see \ref{ss:ex}. More precisely, Theorem \ref{thm:Q} proves that if we look at part of the polynomials consisting of monomials for which the exponent of the variable associated with the `central' vertex of the graph (cf. \ref{ss:graph}) is non-negative, then they agree. (See \ref{ss:str} for precise formulation.) In fact, by Proposition \ref{prop:can}, for the canonical class $h=0$ these are the only monomials, hence   
$P_{0}=P^{+}_{0}$.

\subsection*{Acknowledgements}
TL is supported by ERCEA Consolidator Grant 615655 – NMST and also
by the Basque Government through the BERC 2014-2017 program and by Spanish excellence accreditation SEV-2013-0323. Partial financial support to ZsSz was provided by the `Lendület' program of the Hungarian Academy of Sciences.

\section{Preliminaries}



\subsection{Links of normal surface singularities}\label{ss:link}
\subsubsection{}
Let $\Gamma$ be a connected negative definite plumbing graph with vertices $\mathcal{V}=\mathcal{V}(\Gamma)$. By plumbing disk bundles along $\Gamma$, we obtain a smooth 4--manifold $\widetilde{X}$ whose boundary is an oriented 
plumbed 3--manifold $M$. $\Gamma$ can be realized as the dual graph of a good resolution $\pi:\widetilde{X}\to X$ of some complex normal surface singularity $(X,o)$ 
and $M$ is called the link of the singularity. In our study, we assume that $M$ is a {\it rational homology sphere}, or, equivalently, $\Gamma$ is a tree and all the genus decorations are zero.

Recall that $L:=H_2 (\widetilde{X},\mathbb{Z} )\simeq \mathbb{Z}\langle E_v\rangle_{v\in\mathcal{V}}$ is a lattice, freely
generated by the classes of the irreducible exceptional divisors 
$\{E_v\}_{v\in\mathcal{V}}$ (ie. classes of $2$-spheres), with 
a nondegenerate negative definite intersection form $I:=[(E_v,E_w)]_{v,w\in \mathcal{V}}$. $L':=H^2( \widetilde{X},\mathbb{Z})\simeq Hom(L,\mathbb{Z})$ is the dual lattice, freely generated by the (anti)duals $\{E_v^*\}_{v\in \mathcal{V}}$. $L$ is embedded in $L'$ by the intersection form (which extends to $L\otimes \mathbb{Q}\supset L'$) and their finite quotient is $H:=L'/L \simeq H^2(\partial \widetilde{X},\mathbb{Z})\simeq H_{1}(M, \mathbb{Z})$. 

\subsubsection{}\label{ss:det}
The {\it determinant} of a subgraph $\Gamma'\subseteq \Gamma$ is defined as the determinant of the negative of the submatrix of $I$ with rows and columns indexed with vertices of $\Gamma'$, and it will be denoted by $\mathrm{det}_{\Gamma'}$. In particular, $\mathrm{det}_\Gamma:=\det(-I)=|H|$. We will also consider the following subgraphs: since $\Gamma$ is a tree, for any two vertices $v,w\in \mathcal{V}$ there is a unique minimal connected subgraph 
$[v,w]$ with vertices $\{v_{i}\}_{i=0}^{k}$ such that $v=v_{0}$ and $w=v_{k}$. Similarly, we also introduce notations 
$[v,w)$, $(v,w]$ and $(v,w)$ for the complete subgraphs with vertices $\{v_{i}\}_{i=0}^{k-1}$, $\{v_{i}\}_{i=1}^{k}$ and $\{v_{i}\}_{i=1}^{k-1}$ respectively.

The inverse of $I$ has entries
 $(I^{-1})_{vw}=(E_v^*,E^*_w)$, all of them are negative. Moreover,
 they can be computed using determinants of subgraphs as  (cf. \cite[page 83]{EN})
\begin{equation}\label{eq:DETsgr}
- (E_v^*,E^*_w) = \frac{\det_{\Gamma\setminus [v,w]}}{\det_{\Gamma}}.
\end{equation}

\subsubsection{}\label{ss:order} We can consider the following partial order on $L\otimes \mathbb{Q}$: for any $l_1,l_2$ one writes $l_1\geq l_2$ if
$l_1-l_2=\sum_{v\in \mathcal{V}} \ell_v E_v$ with all $\ell_v\geq 0$. 
The Lipman (anti-nef) cone $\mathcal{S}'$ is defined by $\{l'\in L'\,:\, (l',E_v)\leq 0 \ \mbox{for all
$v$}\}$ and it is generated over $\mathbb{Z}_{\geq 0}$ by the
elements $E_v^*$. We use notation $\mathcal{S}'_{\mathbb{R}} := \mathcal{S}'\otimes \mathbb{R}$ for the real Lipman cone. 


\subsubsection{}
Let $\widetilde{\sigma}_{can}$ be the {\it canonical $spin^c$-structure} on $\widetilde{X}$.  Its
first Chern class $c_1( \widetilde{\sigma}_{can})=-K\in L'$, where $K$ is the canonical class in $L'$ defined by the 
adjunction formulas $(K+E_v,E_v)+2=0$  for all $v\in\mathcal{V}$.
The set of $spin^c$-structures $\mathrm{Spin}^c(\widetilde{X})$ of $\widetilde{X}$ is an $L'$-torsor, ie. if we denote
the $L'$-action by $l'*\widetilde{\sigma}$, then $c_1(l'*\widetilde{\sigma})=c_1(\widetilde{\sigma})+2l'$. 
Furthermore,  all the $spin^c$-structures of $M$ are obtained by restrictions from $\widetilde{X}$.
$\mathrm{Spin}^c(M)$ is an $H$-torsor, compatible with the restriction and the projection $L'\to H$.
The {\it canonical $spin^c$-structure} $\sigma_{can}$ of $M$ is the restriction 
of the canonical $spin^c$-structure $\widetilde{\sigma}_{can}$ of $\widetilde{X}$. Hence, for any $\sigma\in \mathrm{Spin}^c(M)$ one has $\sigma=h*\sigma_{can}$ for some $h\in H$.


\subsection{Seiberg--Witten invariants of normal surface singularities}\label{ss:sw}
For any closed, oriented and connected 3-manifold $M$ we consider the {\it Seiberg--Witten invariant} $\mathfrak{sw}:\mathrm{Spin}^c(M)\rightarrow \mathbb{Q}$, $\sigma\mapsto \mathfrak{sw}_{\sigma}(M)$. In the case of rational homology spheres, it is the signed count of the solutions of the `3-dimensional' Seiberg--Witten equations, modified by the Kreck--Stolcz invariant (cf. \cite{Lim, Nic04}). 

Since its calculation is difficult by the very definition, several topological/combinatorial interpretations have been invented in the last decades. Eg., \cite{Nic04}  has showed that for rational homology spheres $\mathfrak{sw}(M)$ is 
equal with the Reidemeister--Turaev torsion normalized by the Casson--Walker invariant which, in some plumbed cases, can be expressed in terms of the graph and Dedekind--Fourier sums (\cite{Lescop,NN1}). Furthermore, there exist surgery formulas coming from homology exact sequences (eg. Heegaard--Floer homology, monopole Floer homology, lattice cohomology, etc.), where the involved homology theories appear as categorifications of the (normalized) Seiberg--Witten invariant. 

In the case when $M$ is a rational homology sphere link of a normal surface singularity $(X,o)$, different type of surgery- (\cite{BN,LNN}) and combinatorial formulas (\cite{LN,LSz}) have been proved expressing the strong connection of the Seiberg--Witten invariant and the zeta-function/Poincar\'e series associated with $M$ (\cite{NJEMS}). This connection will be explained in the next section. Moreover, we emphasize that the Seiberg--Witten invariant plays a crucial role in the intimate relationship between the topology and geometry of normal surface singularities since it can be viewed as the topological `analogue' of the geometric genus of $(X,o)$, cf. \cite{NN1}.  

For different purposes we may use different normalizations of the Seiberg--Witten invariant. The one we will consider in this article is the following: for any class $h\in H=L'/L$ we define the unique element $r_h\in L'$ characterized by $r_h\in \sum_{v}[0,1)E_v$ with $[r_h]=h$, then 
\begin{equation}\label{swnorm}
\mathfrak{sw}^{norm}_h(M):=-\frac{(K+2r_h)^2+|\mathcal{V}|}{8}-\mathfrak{sw}_{-h*\sigma_{can}}(M)
\end{equation} 
is called the {\em normalized Seiberg--Witten invariant} of $M$ associated with $h\in H$.

\subsection{Zeta-functions and Poincar\'e series}\label{s:ps}
\subsubsection{\bf Definitions and motivation}\label{ss:def}
We have already defined in section \ref{intro2} the multivariable zeta-function $f(\mathbf{t})$ associated with the manifold $M$. Its multivariable Taylor expansion at the origin 
$Z(\mathbf{t})=\sum_{l'}p_{l'} \mathbf{t}^{l'} \in \mathbb{Z}[[L']]$ is called the {\em topological Poincar\'e series}, where $\mathbb{Z}[[L']]$ is the $\mathbb{Z}[L']$-submodule of $\mathbb{Z}[[t_{v}^{\pm 1/|H|}:v\in \mathcal{V}]]$ consisting of series $\sum_{l'\in L'}a_{l'}\mathbf{t}^{l'}$ with $a_{l'}\in \mathbb{Z}$ for all $l'\in L'$. It decomposes naturally into $Z(\mathbf{t})=\sum_{h\in H} Z_{h}(\mathbf{t})$, where $Z_{h}({\mathbf{t}})=\sum_{[l']=h} p_{l'} \mathbf{t}^{l'}$. 
By (\ref{ss:order}), $Z(\mathbf{t})$ is supported in $\mathcal{S}'$, hence $Z_{h}(\mathbf{t})$ is supported in $(l'+L)\cap \mathcal{S}'$, where $l'\in L'$ with $[l']=h$. This decomposition induces a decomposition $f(\mathbf{t})=\sum_{h\in H}f_h(\mathbf{t})$ on the zeta-function level as well, where explicit formula for $f_h(\mathbf{t})$ is provided by \cite{LSznew}.

The zeta-function and its series were introduced by the work of N\'emethi \cite{NPS}, motivated by singularity theory. For a normal surface singularity $(X,o)$ with fixed resolution graph $\Gamma$ we may consider the equivariant divisorial Hilbert series $\mathcal{H}(\mathbf{t})$ which can be connected with the topology of the link $M$ by  introducing the series $\mathcal{P}(\mathbf{t})=
-\mathcal{H}( \mathbf{t}) \cdot \prod_{v\in \mathcal{V}}(1 - t_v^{-1})\in \mathbb{Z}[[L']]$. The point is that, for $h=0$, $Z_0(\mathbf{t})$ serves as the  
`topological candidate' for $\mathcal{P}(\mathbf{t})$: they agree for several class of singularities, eg. for splice quotients (see \cite{NCL}), which
contain all the rational, minimally elliptic or weighted homogeneous singularities. 

For more details regarding to this theory we refer to \cite{CDGPs,CDGEq,NPS,NCL}.

\subsubsection{\bf Counting functions, Seiberg--Witten invariants and reduction}\label{s:sw}
For any $h\in H$ we define the {\it counting function} of the coefficients of $Z_{h}(\mathbf{t})=\sum_{[l']=h}p_{l'} 
\mathbf{t}^{l'}$ by $x\mapsto Q_{h}(x):=\sum_{l'\not\geq x,\, [l']=h} \, p_{l'}.$ This sum is finite since 
$\{l'\in \mathcal{S}'\,:\, l'\ngeq x\}$ is finite 
by \ref{ss:order}. 

Its relation with the Seiberg--Witten invariant is given by a powerful result of N\'emethi \cite{NJEMS} saying that if $x\in (-K+ \textnormal{int}(\mathcal{S}'))\cap L$  then 
\begin{equation}\label{eq:countf}
Q_{h}(x)=\chi_{K+2r_h}(x)+\mathfrak{sw}_h^{norm}(M),
\end{equation}
where $\chi_{K+2r_h}(x):=-(K+2r_h+x,x)/2$. Thus, $Q_{h}(x)$ is a multivariable quadratic polynomial on $L$ with constant term $\mathfrak{sw}^{norm}_h(M)$. 
Furthermore, the idea of the general framework given by \cite{LN} is the following: there exists a conical chamber decomposition of the real cone $\mathcal{S}'_{\mathbb{R}}=\cup_{\tau}\mathcal{C}_{\tau}$, a sublattice $\widetilde L\subset L$ and $l'_* \in \mathcal{S}'$ such that $Q_h(l')$ is a polynomial on  
$\widetilde L\cap(l'_* +\mathcal{C_{\tau}})$, say $Q^{\mathcal{C}_{\tau}}_h(l')$. This allows to define the {\em multivariable periodic constant} by $\mathrm{pc}^{\mathcal{C}_{\tau}}(Z_h):= Q^{\mathcal{C}_{\tau}}_h(0)$ 
associated with $h\in H$ and $\mathcal{C}_{\tau}$. Moreover, $Z_h(\mathbf{t})$ is rather special in the sense that all $Q^{\mathcal{C}_{\tau}}_h$ are equal for any $\mathcal{C}_{\tau}$. In particular, we say that there exists the periodic constant  
$\mathrm{pc}^{S'_{\mathbb{R}}}(Z_h):=\mathrm{pc}^{\mathcal{C}_{\tau}}(Z_h)$ associated with $S'_{\mathbb{R}}$, and in fact, it is equal with $\mathfrak{sw}^{norm}_h(M)$.

%
%
We also notice that (\ref{eq:countf}) has a geometric analogue which 
expresses the geometric genus of the complex normal surface singularity $(X,o)$ from the series $\mathcal{P}(\mathbf{t})$ (cf. \cite{NCL}).

\cite{LN} has showed also that from the point of view of the above relation the number of variables of the zeta-function (or Poincar\'e series) can be reduced to the number of nodes $|\mathcal{N}|$. Thus, if we define the \emph{reduced zeta-function} and \emph{reduced Poincar\'e series} by 
$$
f_h(\mathbf{t}_{\mathcal{N}}) = f_{h}(\mathbf{t})\mid_{t_v=1,v\notin \mathcal{N}}
\qquad\textnormal {and } \qquad
Z_h(\mathbf{t}_{\mathcal{N}}):=Z_h(\mathbf{t})\mid_{t_v=1,v\notin \mathcal{N}},
$$
then there exists the periodic constant of $Z_h(\mathbf{t}_{\mathcal{N}})$ associated with the projected real Lipman cone  
$\pi_{\mathcal{N}}(S'_{\mathbb{R}})$, where $\pi_{\mathcal{N}}:\mathbb{R}\langle E_v\rangle_{v\in\mathcal{V}}\to \mathbb{R}\langle E_v\rangle_{v\in\mathcal{N}}$ is the natural projection along the linear subspace $\mathbb{R}\langle E_v\rangle_{v\notin\mathcal{N}}$, and 
$$\mathrm{pc}^{\pi_{\mathcal{N}}(S'_{\mathbb{R}})}(Z_h(\mathbf{t}_{\mathcal{N}}))=\mathrm{pc}^{S'_{\mathbb{R}}}(Z_h(\mathbf{t}))=
\mathfrak{sw}^{norm}_h(M).$$ We set notation $\mathbf{t}_{\mathcal{N}}^{x} := \mathbf{t}^{\pi_{\mathcal{N}}(x)}$ for any $x\in L'$.  

The above identity allows us to consider only the reduced versions in our study, which has several advantages: the number of reduced variables is drastically smaller, hence reduces the complexity of the calculations; reflects to the complexity of the manifold $M$ (e.g. one-variable case is realized for Seifert 3-manifolds); for special classes of singularities the reduced series can be compared with certain geometric series (or invariants), cf. \cite{NPS}.

\subsection{`Polynomial-negative degree part' decomposition}\label{ss:polSW}

\subsubsection{\bf One-variable case}\label{ss:onenode}
Let $s(t)$ be a one-variable rational function of the form $B(t)/A(t)$ with $A(t)=\prod_{i=1}^d(1-t^{a_i})$ and $a_i>0$. 
Then by \cite[7.0.2]{BN} one has a unique decomposition $s(t)=P(t)+s^{neg}(t)$, where $P(t)$ is a polynomial and $s^{neg}(t)=R(t)/A(t)$ has negative degree with vanishing periodic constant.  Hence, the periodic constant $\mathrm{pc}(s)$ (associated with the Taylor expansion of $s$ and the cone $\mathbb{R}_{\geq0}$) equals $P(1)$. $P(t)$ is called the \emph{polynomial part} while the rational function $s^{neg}(t)$ is called the \emph{negative degree part} of the decomposition. The decomposition can be deduced easily by the following division on the individual rational fractions:
\begin{equation}\label{eq:div}
\frac{t^{b}}{\prod_i(1-t^{a_i})}=-\frac{t^{b-a_{i_0}}}{\prod_{i\neq i_0}(1-t^{a_i})} + \frac{t^{b-a_{i_0}}}{\prod_{i}(1-t^{a_i})}=\sum_{\substack{x_i\geq 1\\ \sum_i x_ia_i\leq b } } p_{(x_i)}\cdot t^{ b-\sum_i x_ia_i} +  \substack{ \textnormal{negative degree} \\ \\\textnormal{rational function} },
\end{equation}
for some coefficients $p_{(x_i)}\in \mathbb{Z}$.
\subsubsection{\bf Multivariable case}\label{ss:twonode}
The idea towards to the multivariable generalization goes back to the theory developed in \cite{LN}, saying that the counting functions associated with zeta-functions are Ehrhart-type quasipolynomials inside the chambers of an induced chamber-decomposition of $L\otimes\mathbb{R}$. Moreover, the previous one-variable division can be generalized to two-variable functions of the form  $s(\mathbf{t})=B(\mathbf{t})/(1-\mathbf{t}^{a_1})^{d_1}(1-\mathbf{t}^{a_2})^{d_2}$ with $a_i>0$. In particular, for  $f_h(\mathbf{t}_{\mathcal{N}})$ viewed as a function in variables $t_n$ and $t_{n'}$, where $n,n'\in\mathcal{N}$ and there is an edge $\overline{nn'}$ connecting them in $\Gamma^{orb}$ (see \cite[Section 4.5]{LN} and \cite[Lemma 19]{LSz}). 
%
%

For more variables, the direct generalization using a division principle for the individual rational terms seems to be hopeless 
because the (Ehrhart) quasipolynomials associated with the counting functions can not be controlled inside the difficult chamber decomposition of $\mathcal{S}'_{\mathbb{R}}$. 

Nevertheless, the authors in \cite{LSz} have proposed a decomposition 
\begin{equation}
f_h(\mathbf{t}_{\mathcal{N}})=P_h(\mathbf{t}_{\mathcal{N}})+f^{-}_h(\mathbf{t}_{\mathcal{N}})
\end{equation}
which defines the polynomial part as 
\begin{equation}\label{eq:polpartdef}
P_{h}(\mathbf{t}_{\mathcal{N}}) = \sum_{\overline{nn'} \ edge \ of \ \Gamma^{orb}} P^{n,n'}_{h}(\mathbf{t}_{\mathcal{N}}) - \sum_{n\in \mathcal{N}} (\delta_{n,\mathcal{N}}-1)P_{h}^{n}(\mathbf{t}_{\mathcal{N}}), 
\end{equation}
where $P^n_h(\mathbf{t}_{\mathcal{N}})$ for any $n\in N$ are the polynomial parts given by the decompositions of $f_h(\mathbf{t}_{\mathcal{N}})$ as a one-variable function in $t_n$, while $P^{n,n'}_{h}(\mathbf{t}_{\mathcal{N}})$ are the polynomial parts viewed $f_h(\mathbf{t}_{\mathcal{N}})$ as a two-variable function in $t_n$ and $t_{n'}$ for any $n,n'\in\mathcal{N}$ so that there are connected by an edge in $\Gamma^{orb}$. Then \cite[Theorem 16]{LSz} implies the main property of the decomposition  
\begin{equation}\label{eq:polpsw}
 P_h(1)=\mathfrak{sw}^{norm}_h(M).
\end{equation}

\section{Decomposition by multivariable division and proof of the algorithm}

In this section we prove the algorithm which expresses the general multivariable polynomial part of \cite{LSz} in terms of a multivariable Euclidean division and a multiplicity function.

\subsection{Multivariable Euclidean division}\label{ss:multidivision}
We consider two Laurent polynomials $A(\mathbf{t}_{\mathcal{N}})$ and $B(\mathbf{t}_{\mathcal{N}})$ supported on the lattice $\pi_{\mathcal{N}}(L')$. The partial order  $l_{1}>l_{2}$ if $l_{1}-l_{2}=\sum_{v\in \mathcal{V}}\ell_{v}$ with $\ell_{v}>0$ for all $v\in \mathcal{V}$ on $L\otimes \mathbb{Q}$ induces a partial order on monomial terms and 

we assume that $A(\mathbf{t}_{\mathcal{N}})$ has a unique maximal monomial term with respect to this partial order denoted by $A_{a} \mathbf{t}_{\mathcal{N}}^{a}$ such that $a>0$. 

We introduce the following multivariable Euclidean division algorithm. We start with quotient $C=0$ and remainder $R=0$. For a monomial term $B_{b}\mathbf{t}_{\mathcal{N}}^{b}$ of $B(\mathbf{t}_{\mathcal{N}})$ if $b\not<a$ then we subtract $(B_{b}\mathbf{t}_{\mathcal{N}}^{b}/A_{a}\mathbf{t}_{\mathcal{N}}^{a})\cdot A(\mathbf{t}_{\mathcal{N}})$ from $B(\mathbf{t}_{\mathcal{N}})$ and we add $B_{b}\mathbf{t}_{\mathcal{N}}^{b}/A_{a}\mathbf{t}_{\mathcal{N}}^{a}$ to the quotient $C(\mathbf{t}_{\mathcal{N}})$, otherwise we pass $B_{b}\mathbf{t}_{\mathcal{N}}^{b}$ from $B(\mathbf{t}_{\mathcal{N}})$ to the remainder $R(\mathbf{t}_{\mathcal{N}})$. By the assumption on $A(\mathbf{t}_{\mathcal{N}})$ the algorithm terminates in finite steps and gives a unique decomposition
\begin{equation}
B(\mathbf{t}_{\mathcal{N}}) = C(\mathbf{t}_{\mathcal{N}}) \cdot A(\mathbf{t}_{\mathcal{N}}) + R(\mathbf{t}_{\mathcal{N}})
\end{equation}
such that $C(\mathbf{t}_{\mathcal{N}})$ is a supported on $\{l'\in \pi_{\mathcal{N}}(L'):l'\not<0\}$ and $R(\mathbf{t}_{\mathcal{N}})$ is supported on $\{l'\in \pi_{\mathcal{N}}(L'):l'<a\}$.

The following decomposition generalizes the one and two-variable cases.

\begin{lemma}\label{lem:+dec}
For any $h\in H$ there exists a unique decomposition  
\begin{equation}\label{eq:decomp}
f_{h}(\mathbf{t}_{\mathcal{N}}) = P^{+}_{h}(\mathbf{t}_{\mathcal{N}}) + f^{neg}_{h}(\mathbf{t}_{\mathcal{N}}),
\end{equation}
where $P^{+}_{h} (\mathbf{t}_{\mathcal{N}})= \sum_{\beta\in\mathcal{B}_h}p_{\beta}\mathbf{t}_{\mathcal{N}}^{\beta}$ is a Laurent polynomial such that $\beta\not<0$ and $f_{h}^{neg}(\mathbf{t}_{\mathcal{N}})$ is a rational function with negative degree in $t_{n}$ for all $n\in \mathcal{N}$.
\end{lemma}

\begin{proof}
First of all we use the fact that for any $h\in H$ one can write $f_{h}(\mathbf{t}_{\mathcal{N}}) = \mathbf{t}_{\mathcal{N}}^{r_h}\cdot\sum_{\ell}b_{\ell} \mathbf{t}_{\mathcal{N}}^{\ell}/\prod_{n\in \mathcal{N}}(1 - \mathbf{t}_{\mathcal{N}}^{a_{n}})$, where $\ell, a_{n}\in \mathbb{Z}\langle E_{n}\rangle_{n\in \mathcal{N}}$ so that $a_{n} =  \lambda_{n}\pi_{\mathcal{N}}(E^{*}_{n})$ for some $\lambda_{n}>0$, $\ell\in \mathbb{R}_{\geq 0}\langle a_{n}\rangle_{n\in\mathcal{N}}$ and $b_{\ell}\in\mathbb{Z}$ (for more precise formulation see \cite{LSznew}). 
Note that $A(\mathbf{t}_{\mathcal{N}})=\prod_{n\in \mathcal{N}}(1-\mathbf{t}_{\mathcal{N}}^{a_{n}})$ has a unique maximal term $(-1)^{|\mathcal{N}|}\mathbf{t}_{\mathcal{N}}^{\sum_{n\in \mathcal{N}}a_{n}}$ with $\sum_{n\in \mathcal{N}}a_{n}>0$. Thus, by the above multivariable Euclidean division we can write 
\begin{equation}\label{eq:uniqueness}
\mathbf{t}_{\mathcal{N}}^{r_{h}} \sum_{\ell} b_{\ell} \mathbf{t}_{\mathcal{N}}^{\ell} = P^{+}_{h}(\mathbf{t}_{\mathcal{N}}) \cdot \prod_{n\in \mathcal{N}}(1-\mathbf{t}_{\mathcal{N}}^{a_{n}})+R_{h}(\mathbf{t}_{\mathcal{N}})
\end{equation} 
and we set $f^{neg}_{h}(\mathbf{t}_{\mathcal{N}}) := \frac{R_{h}(\mathbf{t}_{\mathcal{N}})}{\prod_{n\in \mathcal{N}}(1-\mathbf{t}_{\mathcal{N}}^{a_{n}})}$.

The uniqueness is followed by the assumptions on $P_{h}^{+}$ and $f_{h}^{neg}$, since (\ref{eq:uniqueness}) can be viewed as a one-variable relation considering other variables as coefficients.

%

\end{proof}

\subsection{Multiplicity and relation to the polynomial part}\label{ss:multiplicity} We will show that the polynomial part can be computed from the multivariable quotient $P^{+}_{h}$ by taking its monomial terms with a suitable multiplicity. 
We start by defining the following type of partial orders $\{\mathcal{N},>\}$.
Choose a node $n_{0}\in \mathcal{N}$ and orient edges of $\Gamma^{orb}$ (cf. \ref{ss:intro3}) towards to the direction of $n_{0}$. This induces a partial order on the set of nodes: $n>n'$ if there is an edge in $\Gamma^{orb}$ connecting them, oriented from $n$ to $n'$. Note that $n_{0}$ is the unique minimal node with respect to this partial order. 

\begin{definition} 
Associated with the above partial order and a monomial $\mathbf{t}_{\mathcal{N}}^{\beta} = \prod_{n\in \mathcal{N}}t_{n}^{\beta_{n}}$, we define the following two sign-functions: 
$\mathfrak{s}_{n}(\beta)=1$ if $\beta_{n}\geq 0$ and $0$ otherwise, respectively, assuming $n>n'$ for some $n,n'\in \mathcal{N}$ we set $\mathfrak{s}_{n>n'}(\beta)=1$ if $\beta_{n}\geq0$ and $\beta_{n'}<0$, and $0$ otherwise.
Finally, we define the \emph{multiplicity function} $\mathfrak{s}(\beta) = \mathfrak{s}_{n_{0}}(\beta) + \sum_{n>n'}\mathfrak{s}_{n>n'}(\beta)$.
\end{definition}
\begin{remark}%
The function $\mathfrak{s}$ does not depend on the above partial orders. This can be checked easily for two partial orders with unique minimal nodes connected by an edge in $\Gamma^{orb}$.
\end{remark}

\begin{thm}\label{lm-0}
Consider the multivariable quotient $P^{+}_{h}(\mathbf{t}_{\mathcal{N}})=\sum_{\beta\in \mathcal{B}_h}p_{\beta}\mathbf{t}_{\mathcal{N}}^{\beta}$  of $f_{h}$. 
Then the polynomial part defined in (\ref{eq:polpartdef}) has the following form
$$
P_{h}(\mathbf{t}_{\mathcal{N}}) = \sum_{\beta\in \mathcal{B}_h}\mathfrak{s}(\beta) p_{\beta}\mathbf{t}_{\mathcal{N}}^{\beta}.
$$

\end{thm}
\begin{proof}

Recall that the polynomial part $P_{h}(\mathbf{t}_{\mathcal{N}})$ is defined by (\ref{eq:polpartdef}) using the polynomials $P_{h}^{n'}(\mathbf{t}_{\mathcal{N}})$ and $P_{h}^{n,n'}(\mathbf{t}_{\mathcal{N}})$ for any $n,n'\in\mathcal{N}$ for which there exists an edge connecting them in $\Gamma^{orb}$. 
Moreover, $P^{n'}_{h}$ and $P^{n,n'}_{h}$ are results of one- and two-variable divisions in variables $t_{n'}$ and $t_{n},t_{n'}$, while considering other variables as coefficients. These divisions can be deduced by the above algorithm if we replace the partial order on $L\otimes \mathbb{Q}$ by the corresponding projections `$<_{n'}$' and `$<_{n,n'}$'. That is, $a<_{n'}b$ and $a<_{n,n'}b$ if $a_{n'}<b_{n'}$ and $a_{n}<b_{n}$, $a_{n'}<b_{n'}$, respectively. Since  $a\not<_{n'}b$ and $a\not<_{n,n'}b$ both  imply $a\not<b$, the monomial terms of $P^{n'}_{h}$ and $P^{n,n'}_{h}$ can be found among  monomial terms of $P^{+}_{h}$, more precisely 
\begin{gather*}
P^{n'}_{h}(\mathbf{t}_{\mathcal{N}}) = \sum_{\substack{\beta\in \mathcal{B}_{h} \\ \beta_{n'}\geq0}} p_{\beta} \mathbf{t}_{\mathcal{N}}^{\beta} = \sum_{\beta \in \mathcal{B}_{h}}\mathfrak{s}_{n'}(\beta) p_{\beta}\mathbf{t}_{\mathcal{N}}^{\beta},
\\
P^{n,n'}_{h}(\mathbf{t}_{\mathcal{N}}) = \sum_{\substack{\beta\in \mathcal{B}_{h} \\ \beta_{n}\textnormal{ or } \beta_{n'}\geq0}} p_{\beta} \mathbf{t}_{\mathcal{N}}^{\beta} = \sum_{\beta \in \mathcal{B}_{h}}(\mathfrak{s}_{n'}(\beta) + \mathfrak{s}_{n>n'}(\beta)) p_{\beta}\mathbf{t}_{\mathcal{N}}^{\beta}
\end{gather*}
(assuming that $n>n'$). Thus,
\begin{align*}
P_{h}(\mathbf{t}_{\mathcal{N}}) 
=& 
\sum_{n>n'} P^{n,n'}_{h}(\mathbf{t}_{\mathcal{N}}) - \sum_{n'\in \mathcal{N}} (\delta_{n',\mathcal{N}}-1)P_{h}^{n'}(\mathbf{t}_{\mathcal{N}}) 
\\
=& 
\sum_{n>n'}  \sum_{\beta\in \mathcal{B}_{h}}(\mathfrak{s}_{n'}(\beta)+\mathfrak{s}_{n>n'}(\beta))p_{\beta}\mathbf{t}_{\mathcal{N}}^{\beta}
 - 
 \sum_{n'\in \mathcal{N}}(\delta_{n',\mathcal{N}}-1)\sum_{\beta\in \mathcal{B}_{h}}\mathfrak{s}_{n'}(\beta)p_{\beta} \mathbf{t}_{\mathcal{N}}^{\beta}\\
=&
\sum_{\beta\in \mathcal{B}_{h}} \Big[ \sum_{n>n'} \mathfrak{s}_{n'}(\beta)+\mathfrak{s}_{n>n'}(\beta) - \sum_{n'\in \mathcal{N}}(\delta_{n',\mathcal{N}}-1)\mathfrak{s}_{n'}(\beta) 
\Big] p_{\beta} \mathbf{t}_{\mathcal{N}}^{\beta}
\\
=&
\sum_{\beta \in \mathcal{B}_{h}} (\mathfrak{s}_{n_{0}}(\beta) + \sum_{n>n'}\mathfrak{s}_{n>n'}(\beta))p_{\beta} \mathbf{t}_{\mathcal{N}} 
=\sum_{\beta \in \mathcal{B}_{h}} \mathfrak{s}(\beta) p_{\beta}\mathbf{t}_{\mathcal{N}},
\end{align*}
since $\#\{n \,|\, n>n' \}=\delta_{n',\mathcal{N}}-1$ for $n'\neq n_{0}$ and $\#\{n \,|\, n>n_{0} \}=\delta_{n_{0},\mathcal{N}}$, where $n_{0}$ is the unique minimal node with respect to the partial order.
\end{proof}

\begin{remark}\label{rk:poly-plus} 
\begin{enumerate}[(i)]
\item\label{poly-plus-i}
For $\beta < 0$  we have $\mathfrak{s}(\beta)=0$, while for $\beta \not< 0$ we have $\mathfrak{s}(\beta)\geq1$. Hence, the multiplicity $\mathfrak{s}(\beta)$ is non-zero for every $\beta\in \mathcal{B}_h$, thus every monomial of $P^{+}_{h}$ appears in $P_{h}$.

\item\label{poly-plus-ii} 
The reduced Poincar\'e series $Z_{h}(\mathbf{t}_{\mathcal{N}})$  is the Taylor expansion of $f_{h}(\mathbf{t}_{\mathcal{N}})$ considering $\mathbf{t}_{\mathcal{N}}$ small. One can think of the `endless' multivariable Euclidean division  as expansion of $f_{h}(\mathbf{t}_{\mathcal{N}})$ considering $\mathbf{t}_{\mathcal{N}}$ large. If we take each term of this latter expansion with multiplicity $\mathfrak{s}$ then we recover $P_{h}$, since terms with negative degree in each $t_{n}$ have zero multiplicity.
\end{enumerate}
\end{remark}

\section{Comparisons, examples and $P^+$}

The aim of this section is to compare the two polynomials $P_{h}(\mathbf{t}_{\mathcal{N}})$ and $P^{+}_{h}(\mathbf{t}_{\mathcal{N}})$, given by the two different decompositions, through crucial classes of negative definite plumbing graphs. 

In case of the first class, when the orbifold graph is a bamboo, we will prove that the two polynomials agree. The second class is also motivated by singularity theory and contains the graphs of the manifolds $S^3_{-p/q}(K)$ where $K\subset S^3$ is the connected sum of algebraic knots. Although this class gives examples when the two polynomials do not agree, their structure can be understood using some specialty of these manifolds.

\subsection{The orbifold graph is a bamboo}\label{ss:orb}

Let $\Gamma$ be a negative definite plumbing graph with set of nodes $\mathcal{N}=\{n_{1},\ldots,n_{k}\}$. In this section we will assume that its orbifold graph $\Gamma^{orb}$ is a \emph{bamboo}, ie. $\Gamma^{orb}$ has no nodes.\\ 
\begin{center}
\begin{tikzpicture}
\def\mydotsize{0.04}
\def\mytextsize{\small}
\node (n1) at (0,0) {};
\node at (0,-0.3) {\mytextsize $n_{1}$};
\draw[fill] (n1) circle (\mydotsize);
\node (n2) at (1,0) {};
\node at (1,-0.3) {\mytextsize $n_{2}$};
\draw[fill] (n2) circle (\mydotsize);
\node (nk-1) at (3,0) {};
\node at (3,-0.3) {\mytextsize $n_{k-1}$};
\draw[fill] (nk-1) circle (\mydotsize);
\node (nk) at (4,0) {};
\node at (4,-0.3) {\mytextsize $n_{k}$};
\draw[fill] (nk) circle (\mydotsize);
\draw (0,0) -- (1.5,0);
\draw (2.5,0) -- (4,0);
\node at (2,0) {$\ldots$};
\end{tikzpicture}
\end{center}%
Then we have the following result:
\begin{thm}\label{thm:bamboo}
If the orbifold graph $\Gamma^{orb}$ is a bamboo then 
$P_{h}(\mathbf{t}_{\mathcal{N}}) = P^{+}_{h}(\mathbf{t}_{\mathcal{N}})$ for any $h\in H$, ie. every monomial term of $P^{+}_{h}(\mathbf{t}_{\mathcal{N}})$ appears in $P_{h}(\mathbf{t}_{\mathcal{N}})$ with multiplicity $1$.
\end{thm}

Denote by $\mathfrak{v}_{i}:=\pi_{\mathcal{N}}(E^{*}_{n_{i}})$ the projected vectors for all $i=1,\ldots,k$. 
When $\Gamma^{orb}$ is a bamboo we can write $f_{h}(\mathbf{t}_{\mathcal{N}})$  as linear combination of fractions of form $\displaystyle \frac{\mathbf{t}_{\mathcal{N}}^{\alpha}}{(1-\mathbf{t}_{\mathcal{N}}^{\lambda_{1}\mathfrak{v}_{1}})(1-\mathbf{t}_{\mathcal{N}}^{\lambda_{k}\mathfrak{v}_{k}})}$ for some $ \alpha \in \mathbb{R}_{\geq0}\langle \mathfrak{v}_{i}\rangle_{i=\overline{1,k}} \cap \mathbb{Z}\langle \pi_{\mathcal{N}}(E^{*}_{v})\rangle_{v\in \mathcal{V}}$ and $\lambda_{1},\lambda_{k}>0$ (cf. \cite{LSznew}).
By the uniqueness of the decomposition (\ref{eq:decomp}) and Theorem \ref{lm-0} it is enough to prove the following proposition.

%

\begin{prop}\label{prop-1}
Let $\alpha\in \mathbb{R}_{\geq0}\langle \mathfrak{v}_{i}\rangle_{i=\overline{1,k}} \cap \mathbb{Z}\langle \pi_{\mathcal{N}}(E^{*}_{v})\rangle_{v\in \mathcal{V}}$ and consider the following fraction $\displaystyle\varphi(\mathbf{t}_{\mathcal{N}})=\frac{\mathbf{t}_{\mathcal{N}}^{\alpha}}{(1-\mathbf{t}_{\mathcal{N}}^{\lambda_{1}\mathfrak{v}_{1}})(1-\mathbf{t}_{\mathcal{N}}^{\lambda_{k}\mathfrak{v}_{k}})}$, $\lambda_{1},\lambda_{k}>0$. 
Then for any monomial $\mathbf{t}_{\mathcal{N}}^{\beta}$ of the quotient $\varphi^{+}$ given by the decomposition $\varphi=\varphi^{+}+\varphi^{neg}$ of Lemma \ref{lem:+dec} one has $\mathfrak{s}(\beta)=1$.
\end{prop}

The main tool in the proof of the proposition will be the following lemma.
\begin{lemma}\label{lm-1}
For any $\beta = \sum_{\ell=1}^{k}\beta_{\ell}E_{n_{\ell}} \in \alpha - \mathbb{R}_{\geq0}\langle \mathfrak{v}_{1}, \mathfrak{v}_{k}\rangle$ with not all $\beta_{\ell}$ negative we have
$$
\beta_{1},\ldots,\beta_{i-1}<0\leq \beta_{i},\ldots,\beta_{j} \geq 0 >\beta_{j+1},\ldots,\beta_{k}
$$  
for some $i,j\in\{1,\ldots,k\}$.
\end{lemma}

We denote by $\mathcal{E}_{i}=\mathcal{E}_{i}(\alpha)$ the intersection $\{\beta=\sum_{\ell=1}^{k}\beta_{\ell}E_{n_{\ell}} \,|\, \beta_{i}=0\}\cap (\alpha - \mathbb{R}_{\geq0}\langle \mathfrak{v}_{1},\mathfrak{v}_{k}\rangle)$ and we consider the parametric line $\beta(t) = t\beta + (1-t)\alpha$, $t\in \mathbb{R}$ connecting $\alpha$ to $\beta$. Whenever $\beta(t)$ crosses $\mathcal{E}_{i}$  as $t$ goes from $0$ to $1$ the sign of $\beta_{i}(t)$ changes from positive to negative. Thus, the order in which $\beta(t)$ crosses $\mathcal{E}_{i}$ determines the order in which $\beta_{i}(t)$'s change sign, consequently determines the sign configuration of $\beta_{i}=\beta_{i}(1)$, $i=1,\ldots,k$.

\begin{lemma}\label{lm-2}
Let $\sigma_{i}=\sigma_{i}(\alpha)$ and $\tau_{i}=\tau_{i}(\alpha)$ be  such that $\alpha-\sigma_{i}\mathfrak{v}_{1} = (\alpha - \mathbb{R}_{\geq0}\mathfrak{v}_{1}) \cap \mathcal{E}_{i}$ and $\alpha-\tau_{i}\mathfrak{v}_{k} = (\alpha - \mathbb{R}_{\geq0}\mathfrak{v}_{k}) \cap \mathcal{E}_{i}$ for any $i=1,\ldots,k$. 
If $\alpha=a_{\ell}\mathfrak{v}_{\ell}$, $a_{\ell}\geq0$ for some $\ell \in\{1,\dots,k\}$ then we have 
$$
\sigma_{1}(\alpha) < \ldots < \sigma_{\ell}(\alpha) = \ldots =\sigma_{k}(\alpha) 
\qquad and\qquad 
\tau_{1}(\alpha)=\ldots=\tau_{\ell}(\alpha) > \ldots > \tau_{k}(\alpha).
$$
Moreover, for general $\alpha \in \mathbb{R}_{\geq0} \langle \mathfrak{v}_{i}\rangle_{i=\overline{1,k}}$ one has $\sigma_{1}(\alpha)\leq \ldots \leq \sigma_{k}(\alpha)$ and $\tau_{1}(\alpha) \geq \ldots \geq \tau_{k}(\alpha)$.
\end{lemma}

\begin{proof}
Note that we have additivity $\tau_{i}(\alpha'+\alpha'') = \tau_{i}(\alpha')+\tau_{i}(\alpha'')$ and $\sigma_{i}(\alpha'+\alpha'') = \sigma_{i}(\alpha')+\sigma_{i}(\alpha'')$, hence we may assume that $\alpha=a_{\ell}\mathfrak{v}_{\ell}$. Moreover,  we will only prove the lemma for $\sigma_{i}$'s. 
The intersection point $\alpha-\sigma_{i}\mathfrak{v}_{1}$ is characterized by $(\alpha-\sigma_{i}\mathfrak{v}_{1},E^{*}_{n_{i}})=0$, whence
$$
\sigma_{i} = \sigma_{i}(a_{\ell}\mathfrak{v}_{\ell}) = \frac{(a_{\ell}\mathfrak{v}_{\ell}, E^{*}_{n_{i}})}{(\mathfrak{v}_{1},E^{*}_{n_{i}})}
=
a_{\ell}\frac{(E^{*}_{n_{\ell}}, E^{*}_{n_{i}})}{(E^{*}_{n_{1}},E^{*}_{n_{i}})}.
$$
Therefore, it is enough to show that
\begin{equation}\label{eq-a1}
\frac{(E^{*}_{n_{\ell}}, E^{*}_{n_{i}})}{(E^{*}_{n_{1}},E^{*}_{n_{i}})} 
<
\frac{(E^{*}_{n_{\ell}}, E^{*}_{n_{i+1}})}{(E^{*}_{n_{1}},E^{*}_{n_{i+1}})}, 
\quad \forall\ i < \ell 
\quad \textnormal{ and } \quad
\frac{(E^{*}_{n_{\ell}}, E^{*}_{n_{i}})}{(E^{*}_{n_{1}},E^{*}_{n_{i}})} 
= 
\frac{(E^{*}_{n_{\ell}}, E^{*}_{n_{i+1}})}{(E^{*}_{n_{1}},E^{*}_{n_{i+1}})}, 
\quad \forall\ i \geq \ell.
\end{equation}
Recall that by (\ref{eq:DETsgr}) $\displaystyle (E^{*}_{v}, E^{*}_{w})=-\frac{\det_{\Gamma\setminus[v,w]}}{\det_{\Gamma}}$ for any vertices $v,w$, hence (\ref{eq-a1}) is equivalent to the following determinantal relations
\begin{equation}\label{eq-a2}
\det\nolimits_{\Gamma\setminus [n_{1},n_{i}]}\det\nolimits_{\Gamma\setminus [n_{i+1},n_{\ell}]}-\det\nolimits_{\Gamma\setminus [n_{1},n_{i+1}]} \cdot \det\nolimits_{\Gamma\setminus [n_{i},n_{\ell}]} 
>0, \quad \forall\ i<\ell,
\end{equation}
and equality for $ i\geq \ell$.

We use the technique of N. Duchon (cf. \cite[Section 21]{EN}) to reduce (\ref{eq-a2})  to the case when $\Gamma$ is a bamboo. To do so, we can remove peripheral edges of a graph in order to simplify graph determinant computations. Removal of such an edge is compensated by adjusting the decorations of the graph. Let $v$ be a vertex with decoration $b_{v}$ and which is connected by an edge only to a vertex $w$ with decoration $b_{w}$. If we remove this edge and replace the decoration of the vertex $w$ by $ b_{w}-b_{v}^{-1}$ then the resulting non-connected graph will be also negative definite and its determinant does not change. Using this technique we remove  consecutively every edges on the legs of $\Gamma$, and denote the resulting decorated graph by $\Gamma'$ which consists of a bamboo -- connecting the nodes $n_{1}$ and $n_{k}$ -- and isolated vertices. Note that $\det\nolimits_{\Gamma\setminus [n_{i},n_{j}]} = \det_{\Gamma'\setminus [n_{i},n_{j}]}$ for all $i,j=1,\ldots,k$.  Moreover, (\ref{eq-a2}) is equivalent with
\begin{equation}\label{eq-a3}
\det\nolimits_{\Gamma'\setminus [n_{1},n_{i}]}\det\nolimits_{\Gamma'\setminus [n_{i+1},n_{\ell}]} - \det\nolimits_{\Gamma'\setminus [n_{1},n_{i+1}]} \cdot \det\nolimits_{\Gamma'\setminus [n_{i},n_{\ell}]}  >0, \quad \forall\ i<\ell,
\end{equation}
and equality for $i\geq \ell$, respectively. From  point of view of (\ref{eq-a3}) we can forget about the isolated vertices of $\Gamma'$, ie. we may assume that $\Gamma'$ is a bamboo. If we denote by $\det'_{[n_{i},n_{j}]}$ the determinant of the graph $[n_{i},n_{j}]$ as subgraph of (the bamboo) $\Gamma'$ then for $i<\ell$ we have
\begin{multline*}
\det\nolimits_{\Gamma'\setminus [n_{1},n_{i}]}\det\nolimits_{\Gamma'\setminus [n_{i+1},n_{\ell}]} 
- 
\det\nolimits_{\Gamma'\setminus [n_{1},n_{i+1}]} \cdot \det\nolimits_{\Gamma'\setminus [n_{i},n_{\ell}]} 
=
 \\
 \det'\nolimits_{[n_{1},n_{i+1})} \cdot \det'\nolimits_{(n_{i},n_{k}]}\cdot \det'\nolimits_{(n_{\ell}, n_{k}]}  
 -  
 \det'\nolimits_{[n_{1},n_{i})} \cdot \det'\nolimits_{(n_{i+1},n_{k}]}\cdot \det'\nolimits_{(n_{\ell}, n_{k}]}
 \\
 =
 \det'\nolimits_{[n_{1},n_{k}]} \cdot \det'\nolimits_{(n_{i},n_{i+1})}\cdot \det'\nolimits_{(n_{\ell}, n_{k}]},
\end{multline*}
where the second equality uses the identity $$\det'\nolimits_{[n_{1},n_{i+1})}\cdot\det'\nolimits_{(n_{i},n_{k}]}=\det'\nolimits_{[n_{1},n_{k}]}\cdot\det'\nolimits_{(n_{i},n_{i+1})}+\det'\nolimits_{[n_{1},n_{i})}\cdot\det'\nolimits_{(n_{i+1},n_{k}]}$$ from \cite[Lemma 2.1.2]{LSznew}. $\Gamma'$ is also negative definite, hence $ \det'\nolimits_{[n_{1},n_{k}]} \cdot \det'\nolimits_{(n_{i},n_{i+1})}\cdot \det'\nolimits_{(n_{\ell}, n_{k}]} >0$ (note that $\det'\nolimits_{(n_{k},n_{k}] }=1$). 
If $i\geq \ell$ then it is easy to see 
\begin{multline*}
\det\nolimits_{\Gamma'\setminus [n_{1},n_{i}]}\det\nolimits_{\Gamma'\setminus [n_{i+1},n_{\ell}]} 
- 
\det\nolimits_{\Gamma'\setminus [n_{1},n_{i+1}]} \cdot \det\nolimits_{\Gamma'\setminus [n_{i},n_{\ell}]} 
=
\\
 \det'\nolimits_{(n_{i},n_{k}]} \cdot \det'\nolimits_{[n_{1},n_{\ell})}\cdot \det'\nolimits_{(n_{i+1}, n_{k}]}  
 -  
 \det'\nolimits_{(n_{i+1},n_{k}]} \cdot \det'\nolimits_{[n_{1},n_{\ell})}\cdot \det'\nolimits_{(n_{i}, n_{k}]} = 0.
\end{multline*}
\end{proof}
We also introduce additional notations $\mathcal{E}_{0} = \mathcal{E}_{0}(\alpha)=\alpha- \mathbb{R}_{\geq0}\mathfrak{v}_{k}$ and $\mathcal{E}_{k+1}=\mathcal{E}_{k+1}(\alpha) =\alpha-\mathbb{R}_{\geq0}\mathfrak{v}_{1}$. Moreover, denote by $\varepsilon_{i,j}=\varepsilon_{i,j}(\alpha) = \mathcal{E}_{i}(\alpha)\cap \mathcal{E}_{j}(\alpha)$ the intersection points of segments $\mathcal{E}_{i}$ and $\mathcal{E}_{j}$.

\begin{lemma}\label{lm-21}
On $\mathcal{E}_{i}(\alpha)$ the intersection points are in the following order: $\varepsilon_{i,0}(\alpha) $, $\ldots$, $\varepsilon_{i,i-1}(\alpha)$, $\varepsilon_{i,i+1}(\alpha)$, $\ldots$,  $\varepsilon_{i,k+1}(\alpha)$ for all $i=0,\ldots,k+1$ and for all $\alpha \in \mathbb{R}_{\geq0}\langle \mathfrak{v}_{i}\rangle_{i=\overline{1,k}}$.
\end{lemma}
\begin{proof}
For $i=0$ and $i=k+1$ the statement is immediate from Lemma \ref{lm-2}. Notice that we have defined $\sigma_{i}=\sigma_{i}(\alpha)$ and $\tau_{i}=\tau_{i}(\alpha)$ such that $\varepsilon_{i,0} = \alpha-\tau_{i}\mathfrak{v}_{k}$ and $\varepsilon_{i,k+1} = \alpha-\sigma_{i} \mathfrak{v}_{1}$. 
If $t_{i,j}=t_{i,j}(\alpha)\in [0,1]$ such that $\varepsilon_{i,j} = (1-t_{i,j})\varepsilon_{i,0} + t_{i,j}\varepsilon_{i,k+1}$, then we have to prove that $t_{i,j}(\alpha) \leq t_{i,j+1}(\alpha)$ for all $j$.  Indeed, the case $\alpha=a_{\ell} \mathfrak{v}_{\ell}$, $a_{\ell}\geq0$ follows directly from the first part of Lemma \ref{lm-2}. Generally, notice first the additivity $\varepsilon_{i,j}(\alpha'+\alpha'') = \varepsilon_{i,j}(\alpha')+\varepsilon_{i,j}(\alpha'')$ (as vectors), hence 
\begin{equation}\label{eq-a4}
t_{i,j}(\alpha'+\alpha'') = \frac{t_{i,j}(\alpha') {\sigma}_{i}(\alpha')+t_{i,j}(\alpha''){\sigma}_{i}(\alpha'')}{{\sigma}_{i}(\alpha')+{\sigma}_{i}(\alpha'')}, 
\end{equation}
which gives the result using $t_{i,j}(a_{\ell}\mathfrak{v}_{\ell}) \leq t_{i,j+1}(a_{\ell}\mathfrak{v}_{\ell}) $ for all $j$ and $\ell$.
%
\end{proof}

\begin{lemma}\label{lm-3}
The bounded region $(\alpha - \mathbb{R}_{\geq0}\langle\mathfrak{v}_{1},\mathfrak{v}_{k}\rangle) \setminus \mathbb{R}_{<0}\langle E_{n}\rangle _{n\in \mathcal{N}}$ is the union of quadrangles between segments $\mathcal{E}_{i}, \mathcal{E}_{i+1}, \mathcal{E}_{j}, \mathcal{E}_{j+1}$
or triangles (degenerated cases). 
These polygons may intersect each other only at the boundary. 
\end{lemma}

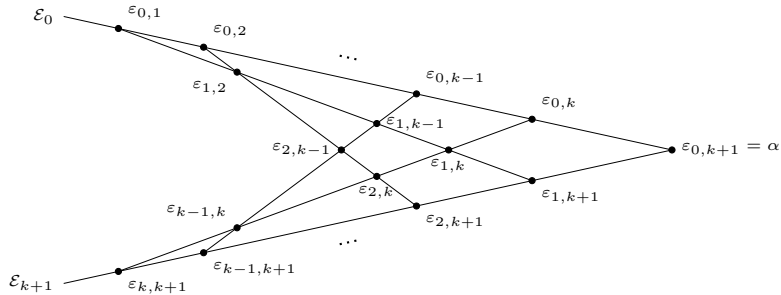
\begin{center}
\begin{tikzpicture}
\def\mydotsize{0.04};
\def\mytextsize{\tiny};
\def\x{-8}
\def\y{1.77}
\coordinate (E0) at ({\x},{\y});
\node[left] at (E0) {\mytextsize $\mathcal{E}_{0}$};
\coordinate (Ek+1) at ({\x},{-\y});
\node[left] at (Ek+1) {\mytextsize $\mathcal{E}_{k+1}$};
\coordinate (e0k+1) at (0,0);
\draw[fill] (e0k+1) circle (\mydotsize);
\node[right] at (e0k+1) {\mytextsize $\varepsilon_{0,k+1}=\alpha$};
\def\sc01{0.91}
\coordinate (e01) at ({\x*\sc01},{\y*\sc01});
\draw[fill] (e01) circle (\mydotsize);
\node[above right] at (e01) {\mytextsize $\varepsilon_{0,1}$};
\coordinate (ekk+1) at ({\x*\sc01},{-\y*\sc01});
\draw[fill] (ekk+1) circle (\mydotsize);
\node[below right] at (ekk+1) {\mytextsize $\varepsilon_{k,k+1}$};
\def\sc01{0.77}
\coordinate (e02) at ({\x*\sc01},{\y*\sc01});
\draw[fill] (e02) circle (\mydotsize);
\node[above right] at (e02) {\mytextsize $\varepsilon_{0,2}$};
\coordinate (ek-1k+1) at ({\x*\sc01},{-\y*\sc01});
\draw[fill] (ek-1k+1) circle (\mydotsize);
\node[below right] at (ek-1k+1) {\mytextsize $\varepsilon_{k-1,k+1}$};
\def\sc01{0.23}
\coordinate (e0k) at ({\x*\sc01},{\y*\sc01});
\draw[fill] (e0k) circle (\mydotsize);
\node[above right] at (e0k) {\mytextsize $\varepsilon_{0,k}$};
\coordinate (e1k+1) at ({\x*\sc01},{-\y*\sc01});
\draw[fill] (e1k+1) circle (\mydotsize);
\node[below right] at (e1k+1) {\mytextsize $\varepsilon_{1,k+1}$};
\def\sc01{0.42}
\coordinate (e0k-1) at ({\x*\sc01},{\y*\sc01});
\draw[fill] (e0k-1) circle (\mydotsize);
\node[above right] at (e0k-1) {\mytextsize $\varepsilon_{0,k-1}$};
\coordinate (e2k+1) at ({\x*\sc01},{-\y*\sc01});
\draw[fill] (e2k+1) circle (\mydotsize);
\node[below right] at (e2k+1) {\mytextsize $\varepsilon_{2,k+1}$};
\draw[very thin] (e0k+1) edge (E0);
\draw[very thin] (e0k+1) edge (Ek+1);
\draw[very thin] (e0k) edge (ekk+1);
\draw[very thin] (e0k-1) edge (ek-1k+1);
\draw[very thin] (e02) edge (e2k+1);
\draw[very thin] (e01) edge (e1k+1);
\coordinate (e1k) at (intersection of {e0k--ekk+1} and {e01--e1k+1});
\draw[fill] (e1k) circle (\mydotsize);
\node[below] at (e1k) {\mytextsize $\varepsilon_{1,k}$};
\coordinate (e1k-1) at (intersection of {e01--e1k+1} and {e0k-1--ek-1k+1});
\draw[fill] (e1k-1) circle (\mydotsize);
\node[right] at (e1k-1) {\mytextsize $\varepsilon_{1,k-1}$};
\coordinate (e12) at (intersection of {e01--e1k+1} and {e02--e2k+1});
\draw[fill] (e12) circle (\mydotsize);
\node[below left] at (e12) {\mytextsize $\varepsilon_{1,2}$};
\coordinate (e2k) at (intersection of {e0k--ekk+1} and {e02--e2k+1});
\draw[fill] (e2k) circle (\mydotsize);
\node[below] at (e2k) {\mytextsize $\varepsilon_{2,k}$};
\coordinate (ek-1k) at (intersection of {e0k--ekk+1} and {ek-1k+1--e0k-1});
\draw[fill] (ek-1k) circle (\mydotsize);
\node[above left] at (ek-1k) {\mytextsize $\varepsilon_{k-1,k}$};
\coordinate (e2k-1) at (intersection of {e02--e2k+1} and {e0k-1--ek-1k+1});
\draw[fill] (e2k-1) circle (\mydotsize);
\node[left] at (e2k-1) {\mytextsize $\varepsilon_{2,k-1}$};
\coordinate (d01) at ($0.565*(E0)+(0.15,0.25)$);
\coordinate (d02) at ($0.535*(E0)+(0.15,0.25)$);
\draw[dotted, thick] (d01)--(d02);
\coordinate (d01) at ($0.565*(Ek+1)+(0.15,-0.25)$);
\coordinate (d02) at ($0.535*(Ek+1)+(0.15,-0.25)$);
\draw[dotted, thick] (d01)--(d02);
\end{tikzpicture}
\end{center}

\begin{proof}
The segments $\mathcal{E}_{i}$ divide $(\alpha - \mathbb{R}_{\geq0} \langle \mathfrak{v}_{1}, \mathfrak{v}_{k}\rangle) \setminus \mathbb{R}_{<0}\langle E_{n}\rangle _{n\in \mathcal{N}} $ into convex polygons. By Lemma \ref{lm-21}, we can assume that $[\varepsilon_{i,j}, \varepsilon_{i,j+1}]$ and $[\varepsilon_{i+1,j},\varepsilon_{i,j}]$ are two faces at vertex $\varepsilon_{i,j}$ of such a polygon. Moreover, $\varepsilon_{i+1,j}$ and $\varepsilon_{i,j+1}$ must be also vertices of the polygon and another two faces must lie on segments $\mathcal{E}_{i+1}$ and $\mathcal{E}_{j+1}$. Hence, the segments $\mathcal{E}_{i}$, $\mathcal{E}_{j}$, $\mathcal{E}_{i+1}$, $\mathcal{E}_{j+1}$ form a convex polygon with vertices $\varepsilon_{i,j}$, $\varepsilon_{i+1,j}$, $\varepsilon_{i,j+1}$ and $\varepsilon_{i+1,j+1}$. The polygon can degenerate into triangles with vertices $\varepsilon_{i,j}$, $\varepsilon_{i,j+1}$ and $\varepsilon_{i+1,j+1}$.
\end{proof}

\begin{proof}[Proof of Lemma \ref{lm-1}]
Let $\beta \in (\alpha - \mathbb{R}_{\geq0} \langle \mathfrak{v}_{1}, \mathfrak{v}_{k}\rangle) \setminus \mathbb{R}_{<0}\langle E_{n}\rangle _{n\in \mathcal{N}}$ be fixed. Consider the parametric line $\beta(t) = t\beta+(1-t)\alpha$ connecting $\beta$ to the vertex $\alpha$ of the affine cone. 
The order in which $\beta(t)$ intersects the segments $\mathcal{E}_{i}$ as $t$ goes from $0$ to $1$  tells us the order in which $\beta_{i}$'s are changing signs. 

In the beginning, every $\beta_{i}>0$ and $\beta(t)$ sits in the polygon with vertices $\alpha=\varepsilon_{0,k+1}$, $\varepsilon_{0,k}$, $\varepsilon_{1,k}$,  $\varepsilon_{k+1,1}$, with sides lying on $\mathcal{E}_{0}$, $\mathcal{E}_{k+1}$, $\mathcal{E}_{1}$, $\mathcal{E}_{k}$. We also say that we have already intersected $\mathcal{E}_{0}$ and $\mathcal{E}_{k+1}$. 
Then $\beta(t)$ either intersects $\mathcal{E}_{1}$, hence $\beta_{1}$ changes to $\beta_{1}<0$ and $\beta(t)$ arrives into the polygon $\varepsilon_{1,k+1}, \varepsilon_{1,k}, \varepsilon_{2,k}, \varepsilon_{2,k+1}$ with sides on $\mathcal{E}_{1}$, $\mathcal{E}_{2}$, $\mathcal{E}_{k}$, $\mathcal{E}_{k+1}$, or, it intersects $\mathcal{E}_{k}$ implying that $\beta_{k}$ becomes negative and $\beta(t)$ arrives into the polygon with sides on $\mathcal{E}_{0}$, $\mathcal{E}_{1}$, $\mathcal{E}_{k-1}$, $\mathcal{E}_{k}$. Therefore, we have crossed $\mathcal{E}_{0}, \mathcal{E}_{1}, \mathcal{E}_{k+1}$ in the first, while $\mathcal{E}_{0}, \mathcal{E}_{k}, \mathcal{E}_{k+1}$ in the second case. 

By induction, we assume that $\beta(t)$ lies in the polygon with sides $\mathcal{E}_{i}, \mathcal{E}_{i+1}, \mathcal{E}_{j}, \mathcal{E}_{j+1}$ for some $t$ and it has already crossed $\mathcal{E}_{0}, \ldots, \mathcal{E}_{i}, \mathcal{E}_{j+1},\ldots, \mathcal{E}_{k+1}$, that is $\beta_{1},\ldots,\beta_{i}, \beta_{j+1},\ldots,\beta_{k}<0$ and $\beta_{i+1},\ldots, \beta_{j}\geq 0$. Thus, $\beta(t)$ must intersects $\mathcal{E}_{i+1}$ or $\mathcal{E}_{j}$. Therefore, either  $\beta_{i+1}$ changes sign to $\beta_{i+1}<0$ and $\beta(t)$ arrives into the polygon with sides $\mathcal{E}_{i+1}, \mathcal{E}_{i+2}, \mathcal{E}_{j}, \mathcal{E}_{j+1}$, or $\beta_{j}$ changes to $\beta_{j}<0$ and $\beta(t)$ arrives into the polygon with $\mathcal{E}_{i}, \mathcal{E}_{i+1}, \mathcal{E}_{j}, \mathcal{E}_{j-1}$. Hence, the induction stops after passing each $\mathcal{E}_{i}$ and proves the desired configuration of signs.
\end{proof}

\begin{proof}[Proof of Proposition \ref{prop-1}]
If $p_{\beta}\mathbf{t}_{\mathcal{N}}^{\beta}$ is a monomial term of $\varphi^{+}(\mathbf{t}_{\mathcal{N}})$ then $\beta \in \alpha - \mathbb{R}_{\geq0}\langle \mathfrak{v}_{1}, \mathfrak{v}_{k}\rangle$, moreover not all $\beta_{\ell}$ are negative and we have sign configuration  as in Lemma \ref{lm-1}. 
To compute the multiplicity $\mathfrak{s}(\beta)$ we choose the ordering of nodes $n_{\ell}>n_{\ell+1}$ for all $\ell=1,\ldots,k-1$. If $\beta_{k}\geq0$ then $\mathfrak{s}_{n_{k}}(\beta)=1$ and $\mathfrak{s}_{n_{\ell}>n_{\ell+1}}(\beta)=0$ for all $\ell=1,\ldots,k-1$, thus $\mathfrak{s}(\beta)=\mathfrak{s}_{n_{k}}(\beta) + \sum_{\ell=1}^{k-1}\mathfrak{s}_{n_{\ell}>n_{\ell+1}}(\beta)=1$. If $\beta_{k}<0$ then $\mathfrak{s}_{n_{k}}(\beta)=0$ and $\mathfrak{s}_{n_{\ell}>n_{\ell+1}}(\beta)=0$ for all $\ell$ except for $\ell=j$, for which $\beta_{j}\geq0$ and $\beta_{j+1}<0$, thus $\mathfrak{s}(\beta)=1$ in this case too.
\end{proof}

\subsection{An example with higher multiplicities}\label{ss:ex} 
Consider the following negative definite plumbing graph $\Gamma$ given by the left hand side of the following picture.

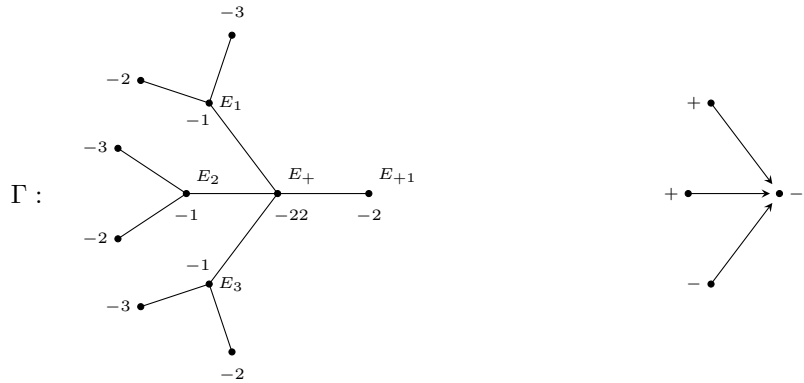
\begin{center}
\begin{tikzpicture}
\def\mydotsize{0.04};
\def\rs{0.6}; 
\coordinate (v2) at ({4*\rs},0);
\draw[fill] (v2) circle (\mydotsize);
\coordinate (v0) at ({6*\rs},0);
\draw[fill] (v0) circle (\mydotsize);
\coordinate (v01) at ({8*\rs},0);
\draw[fill] (v01) circle (\mydotsize);

\coordinate (v1) at ({4.5*\rs},{2*\rs});
\draw[fill] (v1) circle (\mydotsize);
\coordinate (v3) at ({4.5*\rs},{-2*\rs});
\draw[fill] (v3) circle (\mydotsize);
\coordinate (v11) at ({5*\rs},{3.5*\rs});
\draw[fill] (v11) circle (\mydotsize);
\coordinate (v12) at ({3*\rs},{2.5*\rs});
\draw[fill] (v12) circle (\mydotsize);
\coordinate (v21) at ({2.5*\rs},{1*\rs});
\draw[fill] (v21) circle (\mydotsize);
\coordinate (v22) at ({2.5*\rs},{-1*\rs});
\draw[fill] (v22) circle (\mydotsize);
\coordinate (v31) at ({3*\rs},{-2.5*\rs});
\draw[fill] (v31) circle (\mydotsize);
\coordinate (v32) at ({5*\rs},{-3.5*\rs});
\draw[fill] (v32) circle (\mydotsize);

\draw (v0) edge (v01);
\draw  (v0) edge (v1);
\draw  (v1) edge (v11);
\draw  (v1) edge (v12);
\draw  (v0) edge (v2);
\draw  (v0) edge (v3);
\draw  (v2) edge (v21);
\draw  (v2) edge (v22);
\draw  (v3) edge (v31);
\draw  (v3) edge (v32);

\node [above right] at (v0) {\tiny $E_+$};
\node [right] at (v1) {\tiny $E_1$};
\node [above right] at (v01) {\tiny $E_{+1}$};
\node [above right] at (v2) {\tiny $E_2$};
\node [right] at (v3) {\tiny $E_3$};

\node at ({5*\rs},{4*\rs}) {\tiny $-3$};
\node at ({2.5*\rs},{2.5*\rs}) {\tiny $-2$};
\node at ({2*\rs},{1*\rs}) {\tiny $-3$};
\node at ({2*\rs},{-1*\rs}) {\tiny $-2$};
\node at ({2.5*\rs},{-2.5*\rs}) {\tiny $-3$};
\node at ({5*\rs},{-4*\rs}) {\tiny $-2$};
\node at ({6.3*\rs},{-0.5*\rs}) {\tiny $-22$};
\node at ({8*\rs},{-0.5*\rs}) {\tiny $-2$};
\node at ({4.25*\rs},{-1.6*\rs}) {\tiny $-1$}; 
\node at ({4*\rs},{-0.5*\rs}) {\tiny $-1$};
\node at ({4.25*\rs},{1.6*\rs}) {\tiny $-1$}; 

\node at ({0.5*\rs},0) {$\Gamma:$};

\coordinate (u2) at ({15*\rs},0) {};
\node (u0) at ({17*\rs},0) {};
\coordinate (u1) at ({15.5*\rs},{2*\rs});
\coordinate (u3) at ({15.5*\rs},{-2*\rs});

\draw[fill] (u0) circle (\mydotsize);
\draw[fill] (u1) circle (\mydotsize);
\draw[fill] (u2) circle (\mydotsize);
\draw[fill] (u3) circle (\mydotsize);

\draw[-stealth] (u3)--(u0);
\draw[-stealth] (u1)--(u0);
\draw[-stealth] (u2)--(u0);

\node [right] at (u0) {\tiny $-$};
\node [left] at (u1) {\tiny $+$};
\node [left] at (u2) {\tiny $+$};
\node [left] at (u3) {\tiny $-$};
\end{tikzpicture}
\end{center}
The associated plumbed 3-manifold is obtained by $(-7/2)$-surgery along the connected sum of three right handed trefoil knots in $S^3$. Its group $H\simeq \mathbb{Z}_7$ is cyclic of order $7$, generated by the class $[E_{+1}^*]$, where $E_{+1}^*$ is the dual base element in $L'$. For simplicity, we set $\bar{l}:=\pi_{\mathcal{N}}(l)$ for $l\in L\otimes\mathbb{Q}$ and use short notation $(l_+,l_1,l_2,l_3)$ for $\overline{l}=l_+E_+ + \sum_{i=1}^3 l_iE_i$.

Notice that every exponent $\beta=(\beta_+,\beta_1,\beta_2,\beta_3)$ appearing in $P^+(\mathbf{t}_{\mathcal{N}})$ can be written in the form $\beta=c_+\bar{E}^*_+ +\sum_{i=1}^3 c_i \bar{E}^*_i -\sum_{i=1}^3 \sum_{j=1}^2 x_{ij} \bar{E}^*_{ij}-x_{+1}\bar{E}^*_{+1}$ for some $0\leq c_+\leq 2$, $0\leq c_i\leq 1$ and $x_{ij},x_{+1}\geq 1$. Eg, for the choice $c_+=2$, $c_i=1$, $x_{+1}=x_{ij}=1$ for $i\in\{1,2\}$ and $x_{3j}=2$ we get $\beta_0=(-1/7,1/7,1/7,-34/7)$. Moreover, one can check that this is the only way to write $\beta_0$ in the above form. Therefore, the orientation given by the right hand side of the picture above implies that $\mathfrak{s}(\beta_0)=2$. In fact, $\beta_0$ belongs to $P_6^+(\mathbf{t}_{\mathcal{N}})$. 
Hence, by Theorem \ref{lm-0} $$P_6(\mathbf{t}_{\mathcal{N}})\neq P_6^+(\mathbf{t}_{\mathcal{N}}).$$ 

We also emphasize that the exponents
$$(-1/7,1/7,1/7,-34/7), (-1/7,1/7,-34/7,1/7), (-1/7,-34/7,1/7,1/7)  \ \mbox{with coefficient}\  p_{\beta}=1 \ \mbox{and}$$
$$(-1/7,1/7,1/7,-27/7), (-1/7,1/7,-27/7,1/7), (-1/7,-27/7,1/7,1/7) \ \mbox{with coefficient}\  p_{\beta}=-1$$ (all of them from $P_6^+(\mathbf{t}_{\mathcal{N}})$) are the only exponents with $\mathfrak{s}(\beta)=2>1$. Hence, 
although the two polynomials may be different, it still holds that $P_h(1)=P_h^+(1)=\mathfrak{sw}_h^{norm}$ for any $h\in \mathbb{Z}_7$.

\subsection{On the 3-manifold $S^3_{-p/q}(K)$}

\subsubsection{\bf Algebraic knots}\label{sec:algknots} Assume $K\subset S^3$ is an algebraic knot, ie. it is the link of an irreducible plane curve singularity defined by the function germ $\mathfrak{f}:(\mathbb{C}^2,0)\rightarrow (\mathbb{C},0)$. 

The {\it Newton pairs} of $K$ are the pairs
of integers $\{(p_i,q_i)\}_{i=1}^r$, where  $p_i\geq 2$, $q_i\geq
1$, $q_1>p_1$ and gcd$(p_i,q_i)=1$. They are the exponents appearing naturally in the normal form 
of $\mathfrak{f}$. From topological point of view, it is more convenient to use the {\it linking pairs} $(p_i,a_i)_{i=1}^r$ (the decorations of the splice diagram, cf.
\cite{EN}), which can be calculated recursively by 
\begin{equation}\label{eq:linkp}
a_1=q_1 \ \mbox{ and} \  a_{i+1}=q_{i+1}+a_i p_i p_{i+1} \ \mbox{ for }\
i\geq 1.
\end{equation}

The set of intersection multiplicities of $\mathfrak{f}$ with all possible analytic germs is a {\it numerical semigroup} denoted by $\mathcal{M}_\mathfrak{f}$. Although its definition is analytic, $\mathcal{M}_\mathfrak{f}$ is described combinatorially by its Hilbert basis: $p_1p_2\cdots p_r$, $a_{i}p_{i+1}\cdots p_r$ for $1\leq i \leq r-1$, and
$a_r$. In fact, $|\mathbb{Z}_{\geq 0}\setminus \mathcal{M}_\mathfrak{f}|=\mu_\mathfrak{f}/2$ (cf. \cite{Milnorbook}), where $\mu_\mathfrak{f}$ is the Milnor number of $\mathfrak{f}$. The 
Frobenius number of $\mathcal{M}_\mathfrak{f}$ is $\mu_\mathfrak{f}-1$, and  for
$\ell\leq \mu_\mathfrak{f}-1$ one has the symmetry:
\begin{equation}\label{eq:sym}
\ell\in \mathcal{M}_\mathfrak{f} \ \ \mbox{if and only if} \ \ \mu_\mathfrak{f}-1-\ell\not\in \mathcal{M}_\mathfrak{f}. 
\end{equation}
We emphasize that the integer $\delta_\mathfrak{f}:=\mu_\mathfrak{f}/2$ is called the delta-invariant of $\mathfrak{f}$, which equals the minimal Seifert genus of the knot $K$.

The {\it Alexander polynomial} $\Delta(t)$ of $K$ (normalized by $\Delta(1)=1$) can be calculated in terms of the linking pairs via the formula 
\begin{equation}\label{eq:Alex}
\Delta(t)=\frac{ (1-t^{a_1p_1p_2\cdots p_r})(1-t^{a_2p_2\cdots
p_r})\cdots (1-t^{a_r p_r})(1-t)}{(1-t^{a_1p_2\cdots
p_r})(1-t^{a_2p_3\cdots p_r})\cdots (1-t^{a_r})(1-t^{p_1\cdots
p_r})}.
\end{equation}
It has degree $\mu_\mathfrak{f}$. On the other hand, $\Delta(t)/(1-t)=\sum_{\ell\in \mathcal{M}_\mathfrak{f}} t^{\ell}$ is the \emph{monodromy zeta-function} 
of $\mathfrak{f}$  (cf. \cite{CDG99}), whose {\it polynomial part} is calculated explicitly by the gaps of the semigroup: $P_\mathfrak{f}(t)=-\sum_{\ell\notin \mathcal{M}_\mathfrak{f}} t^{\ell}$ (cf. \cite[7.1.2]{LSznew}). Hence, the degree of $P_{\mathfrak{f}}(t)$ equals $\mu_\mathfrak{f}-1$. 

The {\it embedded minimal good resolution graph} of $\mathfrak{f}$ (or the minimal negative-definite plumbing graph of $K$) has the shape of 
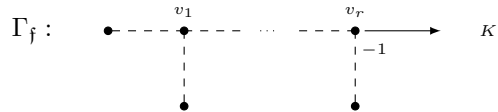
\begin{center}
\begin{tikzpicture}[scale=0.5]
\def\mydotsize{0.1}
\node at (-5.5,1.5) {};
\draw[fill] (-5.5,1.5) circle (\mydotsize);
\node at (-7.5,1.5) {};
\draw[fill] (-7.5,1.5) circle (\mydotsize);
\node at (-5.5,-0.5) {};
\draw[fill] (-5.5,-0.5) circle (\mydotsize);
\draw [dashed] (-5.5,1.5) edge (-7.5,1.5);
\draw [dashed] (-5.5,1.5) edge (-5.5,-0.5);
\draw [dashed] (-5.5,1.5) edge (-4,1.5);
\node at (-4,1.5) {};
\node at (-3.5,1.5) {};
\node at (-3,1.5) {};
\draw [dotted] (-3.5,1.5) edge (-3,1.5);
\node at (-2.5,1.5) {};
\node (v1) at (-1,1.5) {};
\draw[fill] (-1,1.5) circle (\mydotsize);
\node at (-1,-0.5) {};
\draw[fill] (-1,-0.5) circle (\mydotsize);
\node (v2) at (1.5,1.5) {};
\draw [dashed] (-1,1.5) edge (-2.5,1.5);
\draw [dashed] (-1,1.5) edge (-1,-0.5);
\draw [-latex] (v1) edge (v2);
\node at (-5.5,2) {\tiny$v_1$};
\node at (-1,2) {\tiny $v_r$};
\node at (-0.5,1) {\tiny $-1$};
\node at (2.5,1.5) {\tiny $K$};
\node at (-9.5,1.5) {$\Gamma_\mathfrak{f}:$};
\end{tikzpicture}
\end{center}
where the arrowhead, attached to the unique $(-1)$-vertex, represents the knot $K$. Its decorations can be calculated from the Newton pairs $\{(p_i,q_i)\}_{i}$ using eg. \cite{EN}, see also \cite[Section 4.I]{Nded}. The graph has an additional multiplicity decoration: the multiplicity of a vertex is the coefficient of the pullback-divisor of 
$\mathfrak{f}$ along the corresponding exceptional divisor, while the arrowhead
has the multiplicity decoration 1. Eg., we set $m_\mathfrak{f}:=a_r p_r$ to be the multiplicity of the $(-1)$--vertex. 

Notice that the isotopy type of $K\subset S^3$ is completely characterized by any of the following invariants highlighted above. For general references see \cite{BKcurves}, \cite{EN} and also the presentation of \cite{Nded} and \cite{NR}.

\subsubsection{\bf The plumbing of $S^3_{-p/q}(K)$}\label{ss:graph} Let $p/q>0$ ($p>0$, $\gcd(p,q)=1$) be a positive rational number and $\{K_j\}_{j=1}^\nu$ be  a collection of
algebraic knots. Then we consider the oriented 3-manifold $M= S^3_{-p/q}(K)$,
obtained by $(-p/q)$-surgery along the connected sum $K=K_1\#\cdots \#K_\nu\subset S^3$
of the knots $K_j$. All the invariants associated with $K_j$, listed in the previous section, will be indexed by $j$. Eg., the linking pairs of $K_j$ will be denoted by $(p^{(j)}_i,a^{(j)}_i)_{i=1}^{r_j}$, the Alexander polynomial by $\Delta^{(j)}(t)$ and $m^{(j)}$ stands for the multiplicity of the $(-1)$-vertex in the minimal plumbing graph of $K_j$ as above. Set also $m:=\sum_{j=1}^{\nu}m^{(j)}$.  

%
The schematic picture of the plumbing graph $\Gamma$ of the oriented
3--manifold $M=S^3_{-p/q}(K)$ has the following form (cf. \cite{BodNsw}):
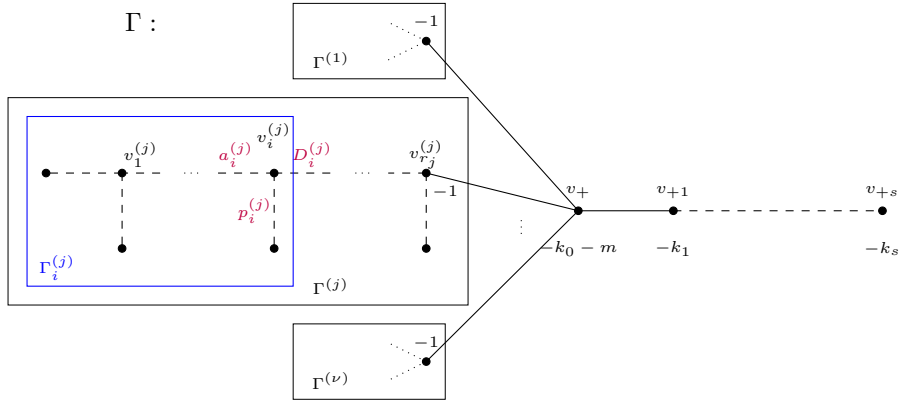
\vspace{0.5cm}
\begin{center}
\begin{tikzpicture}[scale=.5]
\coordinate (v0) at (10.5,0);
\coordinate (vr1) at (6.5,4.5) {};
\coordinate (vrj) at (6.5,1);
\coordinate (vrv) at (6.5,-4);
\coordinate (v01) at (13,0);
\coordinate (v0s) at (18.5,0);

\draw[fill] (v0) circle (0.1);
\draw[fill] (vr1) circle (0.1);
\draw[fill] (vrj) circle (0.1);
\draw[fill] (vrv) circle (0.1);
\draw[fill] (v01) circle (0.1);
\draw[fill] (v0s) circle (0.1);

\draw  (vr1) edge (v0);
\draw  (vrj) edge (v0);
\draw  (vrv) edge (v0);
\draw  (v01) edge (v0);
\draw [dashed] (v01) edge (v0s);
\coordinate (v11) at (5.5,5) {};
\coordinate (v12) at (5.5,4) {} {};
\coordinate (v21) at (5.5,-3.5) {};
\coordinate (v22) at (5.5,-4.5) {};
\draw [dotted] (vr1) edge (v11);
\draw [dotted] (v12) edge (vr1);
\draw [dotted] (vrv) edge (v21);
\draw [dotted] (vrv) edge (v22);
\coordinate (erj) at (6.5,-1) {};
\coordinate (vr1j) at (2.5,1) {} {};
\coordinate (er1j) at (2.5,-1) {} {};
\draw[fill] (erj) circle (0.1);
\draw[fill] (vr1j) circle (0.1);
\draw[fill] (er1j) circle (0.1);

\coordinate (s1) at (1,1) {} {} {};
\coordinate (s1j) at (4,1) {} {} {};
\coordinate (s2) at (0.5,1) {} {} {};
\coordinate (s3) at (0,1) {} {} {};
\coordinate (s4) at (-0.5,1) {} {} {};
\coordinate (s3j) at (4.5,1) {} {} {};
\coordinate (s2j) at (5,1) {} {} {};
\coordinate (s0j) at (5.5,1) {} {} {};

\draw [dashed] (vrj) edge (erj);
\draw [dashed] (vr1j) edge (er1j);
\draw [dashed] (vr1j) edge (s1);
\draw [dashed] (vr1j) edge (s1j);
\draw [dotted] (s2) edge (s3);
\draw [dotted] (s2j) edge (s3j);
\draw [dashed] (vrj) edge (s0j);

\coordinate (v1j) at (-1.5,1) {} {};
\coordinate (e1j) at (-1.5,-1) {} {};
\coordinate (e2j) at (-3.5,1) {} {};
\draw[fill] (v1j) circle (0.1);
\draw[fill] (e1j) circle (0.1);
\draw[fill] (e2j) circle (0.1);

\draw [dashed] (v1j) edge (e2j);
\draw [dashed] (s4) edge (v1j);
\draw [dashed] (v1j) edge (e1j);

\draw [very thin] (3,-3) rectangle (7,-5); 
\draw [very thin] (3,5.5) rectangle (7,3.5); 
\draw [very thin] (-4.5,3) rectangle (7.6,-2.5); 

\node (v3) at (9,0) {};
\node (v4) at (9,-1) {};
\draw [dotted] (v3) edge (v4);

\node at (4,4) {\tiny$\Gamma^{(1)}$};
\node at (4,-2) {\tiny$\Gamma^{(j)}$};
\node at (4,-4.5) {\tiny$\Gamma^{(\nu)}$};
\node at (-1,1.5) {\tiny $v^{(j)}_1$};
\node at (2.5,2) {\tiny $v^{(j)}_{i}$};
\node at (6.5,1.5) {\tiny $v^{(j)}_{r_j}$};
\node at (7,0.5) {\tiny $-1$};
\node at (6.5,5) {\tiny $-1$};
\node at (6.5,-3.5) {\tiny $-1$};
\node at (10.5,0.5) {\tiny $v_+$};
\node at (13,0.5) {\tiny $v_{+1}$};
\node at (18.5,0.5) {\tiny $v_{+s}$};
\node at (10.5,-1) {\tiny $-k_0-m$};
\node at (13,-1) {\tiny $-k_1$};
\node at (18.5,-1) {\tiny $-k_s$};
\node at (-1,5) {$\Gamma:$};

\node at (1.5,1.5) {\color{purple}\tiny $a^{(j)}_i$};
\node at (3.5,1.5) {\color{purple}\tiny $D^{(j)}_i$};
\node at (2,0) {\color{purple}\tiny $p^{(j)}_i$};
\draw [blue] (-4,2.5) rectangle (3,-2); 
\node at (-3.2,-1.5) {\color{blue}\tiny$\Gamma_i^{(j)}$};
\end{tikzpicture} 
\end{center}
\vspace{0.5cm}
where the dash-lines represent strings of vertices. The integers $k_0\geq 1$ and $k_i\geq 2$ $(1\leq i \leq s)$, in the decorations of the vertices $v_{+i}$, are determined by the Hirzebruch/negative continued fraction expansion 
$$ p/q=[k_{0},\ldots, k_s]=
k_{0}-1/(k_{1}-1/(\cdots -1/k_{s})\cdots ).$$ 
We write $E_+$, $E_{+i}$ and $E^{(j)}_{i}$ for the base elements corresponding to the vertices $v_+$, $v_{+i}$ and $v^{(j)}_i$, respectively. 

It is also known that $H=L'/L\simeq \mathbb{Z}_p$ is the cyclic group of order $p$, generated by $[E^*_{+s}]$ (for a complete proof see \cite[Lemma 6]{BodNsw}).

In the above picture we have put at the node $v^{(j)}_{i}$ its splice diagram decorations $a^{(j)}_i$, $p^{(j)}_i$ and $D^{(j)}_i$ (cf. \cite{EN}). Eg., if we use notation 
\begin{equation}\label{eq:Gamma_ij}
\Gamma^{(j)}_i \ \ \mbox{for the subgraph spanned by the nodes $\{v^{(j)}_{i'}\}_{i'=1}^{i}$ and their corresponding end-vertices,}
\end{equation}
then $D^{(j)}_i=\det(\Gamma\setminus \Gamma^{(j)}_i)$. In particular, $\Gamma^{(j)}=:\Gamma^{(j)}_{r_j}$ and its self-intersection decorations are the same as of the embedded minimal good resolution graph of $K_j$, which we omit from the picture for simplicity. 

In the next lemma we prove some useful formulas. 

\begin{lemma}\label{lem:apD}
\begin{enumerate}[(i)]
 \item\label{apD-i} 
 $D^{(j)}_i = p + a^{(j)}_i p^{(j)}_i \left( p^{(j)}_{i+1}\cdots p^{(j)}_{r_{j}} \right)^{2} q,\ $ for $1\leq i\leq r_{j}$;
 \item\label{apD-ii} 
 $a^{(j)}_{i+1}D^{(j)}_i = q^{(j)}_{i+1} p + a^{(j)}_i p^{(j)}_i p^{(j)}_{i+1} D^{(j)}_{i+1},\ $ for $1\leq i\leq r_{j}-1$.
\end{enumerate}
\end{lemma}

\begin{proof}
Let $K'_j$ be the knot with Newton pairs $(p^{(j)}_{i'},q^{(j)}_{i'})_{i'=i+1}^{r_{j}}$.
The graph $\Gamma\setminus \Gamma^{(j)}_i$ is the plumbing graph of the manifold $S^3_{-p'/q}(K'_j\# \ \#_{j'\neq j} K_{j'})$ for some $p'$ which can be computed as follows.  The new linking pairs $(p^{(j)}_{i'},\widetilde{a}^{(j)}_{i'})_{i'=i+1}^{r_{j}}$ can be calculated recursively using (\ref{eq:linkp}) and $\widetilde{a}^{(j)}_{i+1}=q^{(j)}_{i+1}$. Hence, we find the identity  
$$\widetilde{a}^{(j)}_r=a^{(j)}_r - a^{(j)}_i  p^{(j)}_i \left( p^{(j)}_{i+1}\cdots p^{(j)}_{r-1} \right)^{2} p^{(j)}_{r_{j}},$$
which implies that the multiplicity $\widetilde{m}^{(j)}$ of the $(-1)$-vertex in the embedded graph of $K'_j$ equals $\widetilde{a}^{(j)}_{r_{j}} p^{(j)}_{r_{j}} = \widetilde{m}^{(j)}-a^{(j)}_i p^{(j)}_i \left( p^{(j)}_{i+1}\cdots p^{(j)}_{r_{j}} \right)^2$. 
Since the decoration on $v_+$ remains unchanged we must have for the Hirzebruch/negative continued fraction 
$$
p'/q =\big[ k_0+a^{(j)}_i p^{(j)}_i \left( p^{(j)}_{i+1}\cdots p^{(j)}_{r_{j}} \right)^2,k_1,\dots,k_s \big] = p/q+a^{(j)}_ip^{(j)}_i \left( p^{(j)}_{i+1}\cdots p^{(j)}_{r_{j}} \right)^2.
$$ 
Finally, note that $p'=D^{(j)}_{i}$ is the determinant of the graph $\Gamma\setminus \Gamma^{(j)}_{i}$ \cite[Lemma 6]{BodNsw}. This concludes the formula of (\ref{apD-i}) . 
The recursive identity of (\ref{apD-ii}) can be easily verified using (\ref{apD-i}).  
\end{proof}

\subsubsection{\bf Seiberg--Witten invariant via Alexander polynomials}
We consider the product of the Alexander polynomials $\Delta(t):=\prod_{j}\Delta^{(j)}(t)$ with degree $\mu:=\sum_j \mu^{(j)}$. By the known facts $\Delta(1)=1$ and $\Delta'(1)=\mu/2$ we get a unique decomposition 
$$\Delta(t)=1+(\mu/2)(t-1)+(t-1)^2\cdot \mathcal{Q}(t)$$
for some polynomial with integral coefficients $\mathcal{Q}(t)=\sum _{i=0}^{\mu-2}\mathfrak{q}_i t^i$ of
degree $\mu-2$. 

We remark that the coefficients of $\mathcal{Q}$ has many interesting arithmetical properties. Eg., notice that $\mathfrak{q}_0=\mu/2$, $\mathfrak{q}_{\mu-2}=1$ and $\mathfrak{q}_{\mu-2-i} = \mathfrak{q}_i+i+1-\mu/2$ for $0\leq i \leq \mu-2$, given by the symmetry of $\Delta$. The explicit calculation of a general coefficient is rather hard, one can expect it to be connected with some counting function in a semigroup/affine monoid structure associated with the manifold $M$ (cf. \cite{LSznew}). In particular, if $\nu=1$ one can check that $\mathfrak{q}_i=\#\{n \not\in\mathcal{M}\ :\ n>i\}$, where $\mathcal{M}$ is the semigroup of the unique algebraic knot $K$. More details and discussions about these coefficients can be found eg. in \cite{BodN}. 

%
%
%

We look at the decomposition $\mathcal{Q}(t)=\sum_{h\in\mathbb{Z}_p} \mathcal{Q}_h(t)$ where $\mathcal{Q}_h(t):=\sum_{i\geq 0} \mathfrak{q}_{[(ip+h)/q]} t^{[\frac{ip+h}{q}]}$ and consider the following (different) normalization of the Seiberg--Witten invariants: 
\begin{equation} \label{eq:swnorm2}
\widetilde{\mathfrak{sw}}^{norm}_{a}(M):=-\mathfrak{sw}_{-[hE^*_{+s}]\ast\sigma_{can}}(M)- ((K+2hE^*_{+s})^2+\mathcal{V})/8 \ \ \ \ \ \mbox{for} \ \ 0\leq h< p.
\end{equation}
Then the following identity is known by \cite{BN,Ngr,NR}: 
\begin{equation}\label{thm:Q}
\mathcal{Q}_h(1)=\widetilde{\mathfrak{sw}}^{norm}_{h}(M). 
\end{equation} 

\subsubsection{\bf On the structure of the polynomial part}\label{ss:str} 
For any $h\in\mathbb{Z}_p$ consider the decomposition 
$$
f_{h}(\mathbf{t}_{\mathcal{N}}) = P^{+}_{h}(\mathbf{t}_{\mathcal{N}}) + f^{neg}_{h}(\mathbf{t}_{\mathcal{N}}),
$$ 
given by Lemma \ref{lem:+dec}, ie. $f_{h}^{neg}(\mathbf{t}_{\mathcal{N}})$ has negative degree in each variable and write $P^{+}_{h} (\mathbf{t}_{\mathcal{N}})= \sum_{\beta\in\mathcal{B}_h}p_{\beta}\mathbf{t}_{\mathcal{N}}^{\beta}$ where $\beta=(\beta_v)_{v\in\mathcal{N}}$ and $\beta\nless 0$. Let $\beta_+$ be the $E_+$-coefficient of $\beta$ and set $\mathcal{B}:=\bigcup_{h}\mathcal{B}_h$ too. For any polynomial $\mathcal{P}(\mathbf{t}_{\mathcal{N}})$ we consider the decomposition  $\mathcal{P}_{\beta_+\geq0}(\mathbf{t}_{\mathcal{N}})+\mathcal{P}_{\beta_+<0}(\mathbf{t}_{\mathcal{N}})$ so that the first part consists of those monomial terms for which  $\beta_+\geq 0$, and similarly, all the terms of the second part have $\beta_+<0$. 

By definitions we have $P^{+}_{h,\beta_+\geq 0} (\mathbf{t}_{\mathcal{N}})=P^{v_+}_h(\mathbf{t}_{\mathcal{N}})$ and  
Theorem \ref{lm-0} concludes that the monomial terms of $P_h(\mathbf{t}_{\mathcal{N}})$ are exactly of $P^{+}_{h} (\mathbf{t}_{\mathcal{N}})$ with multiplicities.  Therefore, in general, the difference polynomial $\mathcal{D}_h(\mathbf{t}_{\mathcal{N}}):=P_h(\mathbf{t}_{\mathcal{N}})-P^{v_+}_h(\mathbf{t}_{\mathcal{N}})$ consists of 
$P_{h,\beta_+<0}(\mathbf{t}_{\mathcal{N}})$ and the higher multiplicity terms ($\mathfrak{s}(\beta)\geq 2$) from $P_{h,\beta_+\geq0}(\mathbf{t}_{\mathcal{N}})$.

However, in the next theorem we show that there are no monomial terms in $P_{h,\beta_+\geq0}(\mathbf{t}_{\mathcal{N}})$ with $\mathfrak{s}(\beta)\geq 2$, ie. $\mathcal{D}_h(\mathbf{t}_{\mathcal{N}})=P_{h,\beta_+<0}(\mathbf{t}_{\mathcal{N}})$.

\begin{thm}\label{thm2}
$$P_{h,\beta_+\geq0}(\mathbf{t}_{\mathcal{N}})=P^{+}_{h,\beta_+\geq 0} (\mathbf{t}_{\mathcal{N}})=P^{v_+}_h(\mathbf{t}_{\mathcal{N}}).$$ 
\end{thm}

\begin{proof}
First of all we may assume that $\nu\geq 2$ otherwise we have the situation of section \ref{ss:orb}. We fix the orientation of $\Gamma^{orb}$ towards to the node $v_+$ and consider its induced partial order on $\mathcal{N}$ (see \ref{sec:alg}(2)). For any $\beta\in \mathcal{B}$ for which $\beta_+\geq 0$ one has $\mathfrak{s}_{v_+}(\beta)=1$, thus by Theorem \ref{lm-0} we have to prove that $\mathfrak{s}_{v^{(j)}_{i}>v^{(j)}_{i+1}}(\beta)=0$ for any $j\in\{1,\dots, \nu\}$ and $i\in\{1,\dots, r_j\}$. (We set $v^{(j)}_{i+1}:=v_+$.) In order to see this, we prove that the sign configuration on the subgraphs $\Gamma^{(j)}$ behaves exactly as in Lemma \ref{lm-1}. Thus, assuming $\beta_+\geq 0$, for any $j$ we show that
\begin{equation}\label{eq:conf}
\beta^{(j)}_1,\dots,\beta^{(j)}_{i-1}<0\leq \beta^{(j)}_i,\dots,\beta^{(j)}_{r_j},\beta_+ \ \ \mbox{for some} \ \ i\in\{1,\dots, r_j\}.
\end{equation}
Therefore, it is enough to show that Proposition \ref{prop-1} can be applied to the zeta-function $f(\mathbf{t}_{\mathcal{N}_j})$ reduced to the subset of nodes $\mathcal{N}_{j}$ consisting of $v_i^{(j)}$ and $v_+$ of $\Gamma$.

Indeed, for a fixed $j$ we construct a new plumbing graph $\Gamma_{M_j}$ by deleting all the subgraphs $\Gamma^{(j')}$ and its adjacent edges in $\Gamma$ for any $j'\neq j$ and modifying the decoration of $v_+$ into $-k_0-m^{(j)}$. Then the new graph $\Gamma_{M_j}$ is the plumbing graph of the manifold $M_j:=S^3_{-p/q}(K_j)$. Or, if we look at $\Gamma$ as the minimal good resolution graph of a normal surface singularity then one can obtain a new resolution graph by blowing down all the subgraphs $\Gamma^{(j')}$. In this resolution, the new exceptional divisor corresponding to the vertex $v_+$ is a rational curve with singular points and self-intersection $-k_0-m^{(j)}$. If we disregard the singularities of this divisor then we obtain a normal surface singularity whose link is $M_j$ and its minimal good resolution is $\Gamma_{M_j}$.

We distinguish the invariants of the new graphs in the following way: $L_j$ denotes the lattice associated with $\Gamma_{M_j}$ with base elements $E_{v,j}$, the dual lattice will be denoted by $L'_j$ with base elements $E^{*}_{v,j}$. 
We identify $\mathcal{N}_{j}$ of $\Gamma$ with the same set of vertices of $\Gamma_{M_j}$ (notice that $v_+$ is no longer a node in the last graph). Then one can also identify the base elements of $\pi_{\mathcal{N}_{j}}(L)$ and $\pi_{\mathcal{N}_{j}}(L_j)$. In particular, one can show that $\pi_{\mathcal{N}_{j}}(E^*_+)=\pi_{\mathcal{N}_{j}}(E^{*}_{+,j})$. 

Using the above identifications and formula (\ref{eq:Alex}) for Alexander polynomials one can check the following identity  
$$f(\mathbf{t}_{\mathcal{N}_j})=f_{j}(\mathbf{t}_{\mathcal{N}_j})\prod_{j'\neq j} \Delta^{(j')}(\mathbf{t}_{\mathcal{N}_j}^{E^*_+}),$$
where $f_{j}(\mathbf{t}_{\mathcal{N}_j})$ is the zeta-function associated with $\Gamma_{M_j}$ restricted to $\mathcal{N}_j$. The only problem is that $\mathcal{N}_j$  contains $v_+$ which is no longer a node in $\Gamma_{M_j}$. Nevertheless, 
we can blow up the vertex $v_+$ and denote the new graph by $\Gamma'_{M_j}$. Then the newly created $(-1)$-vertex is connected to $v_+$ (if $q=1$ then we can create two such $(-1)$-vertices), hence $v_+$ becomes a node of $\Gamma'_{M_j}$. Using the natural identifications we have $\pi_{\mathcal{N}_j}(E^{*}_{+,j})=\pi_{\mathcal{N}_j}(E^{*}_{b,j})$ where $E^{*}_{b,j}$ denotes the newly created dual base element. Moreover, one has $f_{j}(\mathbf{t}_{\mathcal{N}_j})=f'_{j}(\mathbf{t}_{\mathcal{N}_j})$, where $f'_j$ is associated with $\Gamma'_{M_j}$. 

Finally, the rational function $f'_j(\mathbf{t}_{\mathcal{N}_j})\prod_{j'\neq j} \Delta^{(j')}(\mathbf{t}_{\mathcal{N}_j}^{E^*_+})$ is the sum of rational fractions as in Proposition \ref{prop-1} which implies the sign configuration (\ref{eq:conf}) by Lemma \ref{lm-1}. 

%
\end{proof}

We notice that for the difference polynomial $\mathcal{D}_h(\mathbf{t}_{\mathcal{N}}):=P_h(\mathbf{t}_{\mathcal{N}})-P^{v_+}_h(\mathbf{t}_{\mathcal{N}})$ one has 
$$
\mathcal{D}_h(1)=\mathfrak{sw}^{norm}_{h}(M)-\widetilde{\mathfrak{sw}}^{norm}_{h}(M)=\chi(r_{[hE^*_{+s}]})-\chi(hE^*_{+s}),
$$
where $\chi(l'):=-(K+l',l')/2$ for any $l'\in L'$. This follows from (\ref{eq:polpsw}), (\ref{thm:Q}) and the fact that $P^{v_+}_h(t)=\mathcal{Q}_{h}(t)$, which is proven in \cite[8.1]{BN}. Thus, Theorem \ref{thm2} implies that $P_{h,\beta_+<0}$ counts only the difference between the normalizations and the Seiberg--Witten information is contained in $P^{+}_{h,\beta_+\geq 0}$. 

\subsubsection{\bf Canonical case $h=0$} 
From geometric point of view the main interest focuses to the case when $h=0$, since $f_0(\mathbf{t}_{\mathcal{N}})$ is related with analytic Poincar\'e series associated with a normal surface singularity whose link is $M$ (cf. Section \ref{s:ps}, eg. in the case when $q=1$ the manifold $M=S^3_{-p}(K)$ may appear as the link of a superisolated singularity). 

In this case one has $\mathcal{D}_0(1)=P_{0,\beta_+<0}(1)=0$, although it may happen that there are some monomial terms appearing in $P_{0,\beta_+<0}$. This can indeed occur for $h\neq 0$ as shown by the example from Section \ref{ss:ex}. 
However, in the sequel we prove that for $h=0$ this is not the case, ie. 
$$P^+_{0,\beta_+<0}(\mathbf{t}_{\mathcal{N}})=P_{0,\beta_+<0}(\mathbf{t}_{\mathcal{N}})\equiv 0.$$ Thus, we have $P_0(\mathbf{t}_{\mathcal{N}})=P^+_0(\mathbf{t}_{\mathcal{N}})$, in particular $P^+_{0}(1)=\mathfrak{sw}^{norm}_0(M)$.


\begin{lemma}\label{lem:a0}
Let $\mathfrak{f}^{(j)}_i$ be the irreducible plane curve singularity with Newton pairs $(p^{(j)}_{i'},q^{(j)}_{i'})_{i'=1}^i$ for any $1\leq j\leq \nu$ and $1\leq i\leq r_j$ and its associated semigroup will be denoted by $\mathcal{M}_{\mathfrak{f}^{(j)}_i}$. 
For any $\beta\in \mathcal{B}$, $1\leq j\leq \nu$ and $1\leq i\leq r_j$ we have the following relations
\begin{enumerate}[(i)]
 \item\label{a0-i}
 $$a^{(j)}_{i+1}\beta^{(j)}_{i}=a^{(j)}_{i}p^{(j)}_{i}\beta^{(j)}_{i+1}+q^{(j)}_{i+1}\ell_{\mathfrak{f}^{(j)}_i}^{\beta},$$ where $\ell_{\mathfrak{f}^{(j)}_i}^{\beta} \in \mathbb{Z}\setminus \mathcal{M}_{\mathfrak{f}^{(j)}_i}$ depending on $\beta$. In particular, for $i=r_j$ we set $a^{(j)}_{r_j+1}:=1$, $q^{(j)}_{r_j+1}:=1$ and $\beta^{(j)}_{r_j+1}:=\beta_+$, hence the identity becomes $\beta^{(j)}_{r_j}=m^{(j)}\beta_+ + \ell_{\mathfrak{f}^{(j)}}^{\beta}$.
 \item\label{a0-ii} 
 $$\beta^{(j)}_{i}<a^{(j)}_{i}p^{(j)}_{i}\cdots p^{(j)}_{r_j}(\beta_+ +1).$$
\end{enumerate}
\end{lemma}

\begin{proof}
(\ref{a0-i}) We can write
\begin{equation}\label{eq:sumbet}
\beta=k_+E_+^*+\sum_j(\sum_i k^{(j)}_i E_i^{(j)*}-\sum_{v\in \mathcal{E}^{(j)}}x^{(j)}_v E_v^{(j)*})-x_+E^*_{+s}
\end{equation}
for some integers $0\leq k_+\leq \nu-1$, $k^{(j)}_i\in \{0,1\}$ and $x^{(j)}_v,x_+ \geq 1$, where we use notation $\mathcal{E}^{(j)}$ for the set of end-vertices of $\Gamma^{(j)}$. 

For any subgraph $\Gamma'$ of $\Gamma$ let us denote by $\beta^{\Gamma'}$ the partial sum considering only those terms from the right hand side of (\ref{eq:sumbet}) which  are associated with the nodes and end-vertices of $\Gamma'$. Recall that we have defined the subgraphs $\Gamma^{(j)}_i$ in (\ref{eq:Gamma_ij}). Then, we claim that 
\begin{equation}\label{eq:longcalc}
a^{(j)}_{i+1}\beta^{(j)}_{i}:=a^{(j)}_{i+1}\cdot(\beta,-E_i^{(j)*})=a^{(j)}_{i}p^{(j)}_{i}\beta^{(j)}_{i+1}+\Big(\frac{a^{(j)}_{i+1}D^{(j)}_{i}}{p^{(j)}_{i+1}D^{(j)}_{i+1}}-a^{(j)}_{i}p^{(j)}_{i}\Big)\cdot\beta^{\Gamma^{(j)}_i}_{i+1}.
\end{equation}
For the above expression we have used the following identities: $a^{(j)}_{i+1}(E_v^*,E_i^{(j)*})$ equals either with $\frac{a^{(j)}_{i+1}D^{(j)}_{i}}{p^{(j)}_{i+1}D^{(j)}_{i+1}}(E_v^*,E_{i+1}^{(j)*})$ in the case when $v$ is a node or an end-vertex of $\Gamma^{(j)}_i$, or, with $a^{(j)}_{i}p^{(j)}_{i}(E_v^*,E_{i+1}^{(j)*})$ otherwise. Moreover, one can check from (\ref{eq:sumbet}) that 
\begin{equation}\label{eq:beta-explained}
\beta^{\Gamma^{(j)}_i}_{i+1}=\frac{p^{(j)}_{i+1}D^{(j)}_{i+1}}{p}\Big(\sum_{i'=1}^{i}\big(k_{i'}^{(j)}\cdot a^{(j)}_{i'} p^{(j)}_{i'}\dots p^{(j)}_{i}-x^{(j)}_{v_{i'}}\cdot a^{(j)}_{i'}p^{(j)}_{i'+1}\dots p^{(j)}_{i}\big) - x^{(j)}_{v_{0}}\cdot p^{(j)}_{1}\dots p^{(j)}_{i}\Big). 
\end{equation}
Now, the idea is that by (\ref{eq:beta-explained}) and (\ref{eq:Alex}) the quantity 
$\ell_{\mathfrak{f}^{(j)}_i}^{\beta}:=p/(p^{(j)}_{i+1}D^{(j)}_{i+1})\beta^{\Gamma^{(j)}_i}_{i+1}$ can be viewed as an exponent coming from the division of the monodromy zeta-function of $\mathfrak{f}^{(j)}_i$. Hence, it is either negative or it is an exponent of the polynomial part of the monodromy zeta-function which implies  $\ell_{\mathfrak{f}^{(j)}_i}^{\beta}\notin \mathcal{M}_{\mathfrak{f}^{(j)}_i}$ by \cite{LSznew}. Therefore (\ref{eq:longcalc}) transforms into
$$a^{(j)}_{i+1}\beta^{(j)}_{i}=a^{(j)}_{i}p^{(j)}_{i}\beta^{(j)}_{i+1}+\frac{a^{(j)}_{i+1}D^{(j)}_{i}-a^{(j)}_{i}p^{(j)}_{i}p^{(j)}_{i+1}D^{(j)}_{i+1}}{p} \ell_{\mathfrak{f}^{(j)}_i}^{\beta}=a^{(j)}_{i}p^{(j)}_{i}\beta^{(j)}_{i+1}+q^{(j)}_{i+1}\ell_{\mathfrak{f}^{(j)}_i}^{\beta},$$
where the second equality uses Lemma \ref{lem:apD}(\ref{apD-ii}).

(\ref{a0-ii}) According to the proof of part (\ref{a0-i}) and section \ref{sec:algknots} we can write $\ell_{\mathfrak{f}^{(j)}_i}^{\beta}=\mu_{\mathfrak{f}^{(j)}_i}-1-s_{\mathfrak{f}^{(j)}_i}$ for some $s_{\mathfrak{f}^{(j)}_i}\in \mathcal{M}_{\mathfrak{f}^{(j)}_i}$. Therefore (\ref{a0-i}) implies 
$\beta^{(j)}_{i}\leq (a^{(j)}_{i}p^{(j)}_{i}/a^{(j)}_{i+1})\beta^{(j)}_{i+1}+(q^{(j)}_{i+1}/a^{(j)}_{i+1})(\mu_{\mathfrak{f}^{(j)}_i}-1)$, which induces the following inequality 
\begin{equation}\label{eq:ineq}
\beta^{(j)}_{i} \leq 
a^{(j)}_{i}p^{(j)}_{i}\cdots p^{(j)}_{r_j}\beta_{+} 
+ 
\frac{q^{(j)}_{i+1}}{a^{(j)}_{i+1}}(\mu_{\mathfrak{f}^{(j)}_i}-1)
+
\sum_{i'=i+1}^{r_j}\frac{a^{(j)}_{i}p^{(j)}_{i}\cdots p^{(j)}_{i'-1}q^{(j)}_{i'+1}}{a^{(j)}_{i'}a^{(j)}_{i'+1}}(\mu_{\mathfrak{f}^{(j)}_{i'}}-1).
\end{equation}
We apply (\ref{eq:linkp}) on $q_{i'+1}^{(j)}$ to get
\begin{equation}\label{eq:ineq2}
\beta^{(j)}_{i} \leq 
a^{(j)}_{i}p^{(j)}_{i}\cdots p^{(j)}_{r_j}\beta_{+} 
+ 
(\mu_{\mathfrak{f}^{(j)}_i}-1)
+
\sum_{i'=i+2}^{r_j}\frac{a^{(j)}_{i}p^{(j)}_{i}\cdots p^{(j)}_{i'-1}}{a^{(j)}_{i'}}(\mu_{\mathfrak{f}^{(j)}_{i'}}-1) - \frac{a^{(j)}_{i}p^{(j)}_{i}\cdots p^{(j)}_{i'}}{a^{(j)}_{i'}}(\mu_{\mathfrak{f}^{(j)}_{i'-1}}-1).
\end{equation}
 
We use a well-known recursive formula $\mu_{\mathfrak{f}^{(j)}_{i'}}=(a^{(j)}_{i'}-1)(p^{(j)}_{i'}-1)+p^{(j)}_{i'}\mu_{\mathfrak{f}^{(j)}_{i'-1}}$ (see eg. \cite[(4.13)]{Nded}) for the Milnor numbers of $\mathfrak{f}^{(j)}_{i'}$, which can be rewritten for our purpose in the form
$$
\frac{a^{(j)}_{i}p^{(j)}_{i}\cdots p^{(j)}_{i'-1}}{a^{(j)}_{i'}}(\mu_{\mathfrak{f}^{(j)}_{i'}}-1) 
-
\frac{a^{(j)}_{i}p^{(j)}_{i}\cdots p^{(j)}_{i'}}{a^{(j)}_{i'}}(\mu_{\mathfrak{f}^{(j)}_{i'-1}}-1)
= 
a^{(j)}_{i}p^{(j)}_{i}\cdots p^{(j)}_{i'}
-
a^{(j)}_{i}p^{(j)}_{i}\cdots p^{(j)}_{i'-1}.
$$
Finally, this recursion can be applied repeatedly to (\ref{eq:ineq}) in order to deduce
$$
\beta^{(j)}_{i} 
\leq 
a^{(j)}_{i}p^{(j)}_{i}\cdots p^{(j)}_{r_j}\beta_+ 
+
\sum_{i'={i+1}}^{r_j}(a^{(j)}_{i}p^{(j)}_{i}\cdots p^{(j)}_{i'} - a^{(j)}_{i}p^{(j)}_{i}\cdots p^{(j)}_{i'-1})
+\mu_{\mathfrak{f}^{(j)}_{i}}-1<a^{(j)}_{i}p^{(j)}_{i}\cdots p^{(j)}_{r_j}(\beta_+ +1),
$$
where the second (strict) inequality uses \cite[Theorem 4.12(a)]{Nded}, saying that $m_{\mathfrak{f}^{(j)}_{i}}-\mu_{\mathfrak{f}^{(j)}_{i}}+1\geq 2|\mathcal{V}(\Gamma_{i}^{(j)})|-1>0$.
\end{proof}

\begin{prop}\label{prop:can}
For $h=0$ one has $P_0(\mathbf{t}_{\mathcal{N}})=P^+_0(\mathbf{t}_{\mathcal{N}})=P^{v_+}_0(\mathbf{t}_{\mathcal{N}})$. 
\end{prop}

\begin{proof}
By Theorems \ref{lm-0} and \ref{thm2} we have to show that $P^+_{0,\beta_+<0}(\mathbf{t}_{\mathcal{N}})\equiv 0$. Moreover, using the configuration of signs from the proof of Theorem \ref{thm2}, it needs to be proved that $\beta_+<0$ implies $\beta^{(j)}_{i}<0$ for any $1\leq j\leq \nu$ and $1\leq i\leq r_j$. This is implied by part (\ref{a0-ii}) of Lemma \ref{lem:a0}, since $\beta_+,\beta^{(j)}_{i}\in \mathbb{Z}$ in the case $h=0$. 
\end{proof}


\subsection{Question about $P^+$}

We have shown an example in Section \ref{ss:ex} in which $P_{h}(\mathbf{t}_{\mathcal{N}})\neq P^{+}_{h}(\mathbf{t}_{\mathcal{N}})$ for some $h\in H$. Hence, by Theorem \ref{lm-0} the polynomial part $P_{h}$ in general can be `thicker' than $P^{+}_{h}$. Nevertheless, they have the same set of exponents for the monomials which object presumably plays an important role in geometrical applications. 

On the other hand, the calculation of $P^{+}_{h}$ is much more effective, therefore, it is natural to pose the question whether it can replace $P_{h}$ as a polynomial part. More precisely, we ask the following:
\begin{center}\emph{
Is it true in general that $P^{+}_{h}(1)=\mathfrak{sw}_h^{norm}$? }
\end{center}

\end{document}